\documentclass[12pt]{amsart}
\usepackage{color}
\usepackage{amsmath}
\usepackage{amsthm}
\usepackage{amssymb}
\usepackage{graphicx}
\usepackage{amsfonts}
\usepackage{diagrams}
\usepackage{mathrsfs}
\usepackage{mathdots}
\usepackage{color}

\numberwithin{equation}{section}
\theoremstyle{plain}
\newtheorem{Proposition}[equation]{Proposition}
\newtheorem{Corollary}[equation]{Corollary}
\newtheorem{Theorem}[equation]{Theorem}
\newtheorem{Lemma}[equation]{Lemma}
\theoremstyle{definition}
\newtheorem{Definition}[equation]{Definition}

\newtheorem{Example}[equation]{Example}
\newtheorem{Remark}[equation]{Remark}
\def\HH{\mathscr{H}}

\def\R{\mathbb{R}}
\def\D{\mathbb{D}}
\def\T{\mathbb{T}}
\def\N{\mathbb{N}}

\def\ov{\overline}

\title[Direct and reverse Carleson measures]{Direct and reverse Carleson measures for
$\HH(b)$ spaces}

\author[Blandign\`{e}res]{Alain Blandign\`{e}res}
 \address{Institut Camille Jordan, Universit\'e Claude Bernard Lyon 1,  69622 Villeurbanne C\'edex, France}
 \email{blandigneres@math.univ-lyon1.fr}

\author[Fricain]{Emmanuel Fricain}
 \address{Laboratoire Paul Painlev\'e, Universit\'e Lille 1, 59 655 Villeneuve d'Ascq C\'edex }
 \email{emmanuel.fricain@math.univ-lille1.fr}

\author[Gaunard]{Fr\'{e}d\'{e}ric Gaunard}
\address{Department of Mathematics, KTH Royal Institute of Technology, 10044 Stockholm, Sweden}
\email{gaunard@kth.se}

\author[Hartmann]{Andreas Hartmann}
\address{Institut de Math\'ematiques de Bordeaux, Universit\'e Bordeaux 1, 351 cours de la Lib\'eration 33405 Talence C\'edex, France}
\email{Andreas.Hartmann@math.u-bordeaux1.fr}

\author[Ross]{William T. Ross}
	\address{Department of Mathematics and Computer Science, University of Richmond, Richmond, VA 23173, USA}
	\email{wross@richmond.edu}

\thanks{This work is supported by ANR FRAB: ANR-09-BLAN-0058-02}



\begin{document}

\begin{abstract}
%The classical embedding theorem of Carleson deals with finite positive Borel measures $\mu$ on the closed unit disk for which there exists a positive constant $c$ such that $\|f\|_{L^2(\mu)} \leq c \|f\|_{H^2}$ for all  $f \in H^2$, the Hardy space of the unit disk. Lef\`evre et al.\ examined measures $\mu$ for which there exists a positive constant $c$ such that $\|f\|_{L^2(\mu)} \geq c \|f\|_{H^2}$ for all  $f \in H^2$. The first type of inequality above was explored with $H^2$ 
%replaced by de Branges-Rovnyak spaces $\mathcal{H}(b)$ by Baranov, Fricain, and Mashreighi. In this paper we give a charaterization of measures $\mu$ for which the second type of inequality holds in $\mathcal{H}(b)$ when the symbol $b$ is not extreme and investigate the analogous condition for "direct" embedding.
In this paper we discuss direct and reverse Carleson measures for the de Branges-Rovnyak spaces $\HH(b)$, mainly when $b$ is a non-extreme point of the unit ball of $H^\infty$.
\end{abstract}

\maketitle

\section{Introduction}

In this paper we wish to discuss reverse Carleson measures for the reproducing kernel Hilbert space $\mathscr{H}(b)$ of analytic functions on the open unit disk $\D$ whose  reproducing kernel is given by
$$k^{b}_{\lambda}(z) := \frac{1 - \overline{b(\lambda)} b(z)}{1 - \overline{\lambda} z},\quad \lambda\in\D.$$
Here $b$ belongs to $H^{\infty}_{1}$, the unit ball in $H^{\infty}$, and $H^\infty$ is the Banach algebra of bounded analytic functions on $\D$ normed with the supremum norm $\|\cdot\|_{\infty}$.  The space  $\HH(b)$ is often known as the \emph{de Branges-Rovnyak space} and we will review the basics of this space in a moment. For now, note that when $\|b\|_{\infty} < 1$, then $\mathscr{H}(b)$ is just the classical Hardy space $H^2$ \cite{Duren,Garnett} with an equivalent norm while if $b$ is an inner function, meaning $|b|=1$ almost everywhere on $\T=\partial\D$, then $\mathscr{H}(b)$ is the classical model space $(b H^{2})^{\perp} = H^2 \ominus b H^2$. For any $b$, the space $\mathscr{H}(b)$ is contractively contained in $H^2$. As is often the case, the properties of $\HH(b)$ spaces, including direct and reverse Carleson measures, depend on whether or not $b$ is an extreme point of $H^\infty_1$. Recall \cite{MR0098981} that $b$ is an extreme point of $H^{\infty}_{1}$ when $\int_{\T} \log(1-|b|) dm = - \infty$, where $m$ is normalized Lebesgue measure on the unit circle $\T$. 

Let $M_{+}(\D^{-})$ denote the positive finite Borel measures on the closed unit disk $\D^{-}$. By a reverse Carleson measure for $\mathscr{H}(b)$ we mean a measure $\mu \in M_{+}(\D^{-})$, for which 
$$\|f\|_{b} \lesssim \|f\|_{\mu},$$
where $\|\cdot\|_{\mu}$ represents the standard $L^2(\mu)$ norm, $\|\cdot\|_{b}$ represents the norm in $\mathscr{H}(b)$, and $f$ ranges over a suitable set in $\HH(b)$. We will make this more precise below.

 A direct Carleson measure is, as to be expected and is well studied for many spaces of analytic functions, a  measure $\mu$ for which $\|f\|_{\mu} \lesssim \|f\|_{b}$ for all $f \in \mathscr{H}(b)$, i.e., $\mathscr{H}(b)$ embeds continuously into $L^2(\mu)$. A reader familiar with $\HH(b)$ spaces will know how difficult it can be to compute or even estimate the norm of an element in $\HH(b)$. This question of direct and reverse Carleson measures could be helpful in this direction. 
Note that such measures are also called sampling measures.
%Moreover, for discrete measures
In particular, there is an interesting and close connection with the problem of sampling sequences (see Corollary~\ref{Cor:sampling} below). 

In order to give a more precise defintion of reverse Carleson measure,
the alert reader will have noticed that when the measure $\mu$ has part of the unit circle $\T$ in its carrier, it is not quite clear if the boundary values of every $f \in \mathscr{H}(b)$ exist $\mu$-almost everywhere so that the integrals $\|f\|_{\mu}$ might not make sense. By a {\it carrier} of a measure $\mu\in M_{+}(\D^{-})$ we mean a Borel set $C\subset \D^-$ for which $\mu(A\cap C)=\mu(A)$ for all Borel subsets $A\subset\D^-$. In a way, we want to be as broad as possible as to not impose a too stringent condition on $\mu$, and so we will only require the reverse inequality to hold on a %sufficiently 
dense set in $\mathscr{H}(b)$. To specify this dense set, and to make sure the integrals $\|f\|_{\mu}$ are well defined for $f$ in this dense set, we make the following definition. 

\begin{Definition}\label{defn:admissible}
For $\mu \in M_{+}(\D^{-})$ we say that an analytic function $f$ on $\D$ is $\mu$-\emph{admissible} if the non-tangential limits of $f$ exist $\mu$-almost everywhere on $\T$. We let $\mathscr{H}(b)_{\mu}$ denote the set of $\mu$-admissible functions in $\mathscr{H}(b)$.
\end{Definition}

With this definition in mind, if $f \in \mathscr{H}(b)_{\mu}$, then defining $f$ on the carrier of $\mu|\T$ via its non-tangential boundary values, we see that $\|f\|_{\mu}$ is well defined with a value in $[0, +\infty]$.

Of course if $\mu$ is carried on $\D$, i.e., $\mu(\T) = 0$, then $\HH(b)_\mu=\HH(b)$. So Definition~\ref{defn:admissible} becomes meaningful when $\mu$ has part of the unit circle $\T$ in its carrier. Certainly for (normalized) Lebesgue measure $m$ on $\T$ we know that  $\mathscr{H}(b) = \mathscr{H}(b)_{m}$ (since $\mathscr{H}(b)$ is always contractively contained in $H^2$ and, via standard theory \cite{Duren, Garnett}, $H^2$ functions have non-tangential boundary values $m$-almost everywhere) though there are often other $\mu$, even ones with non-trivial singular parts on $\T$ with respect to $m$, for which $\mathscr{H}(b) = \mathscr{H}(b)_{\mu}$.  The Clark measures associated with an inner function $b$ (which are known to be singular with respect to $m$) have this property (see  \cite{BFGHR, CRM}). 

If $b$ is a $\mu$-admissible function, then so are all the reproducing kernels $k^{b}_{\lambda}$ (along with finite linear combinations of them) and thus, with this admissibility assumption on $b$, $\mathscr{H}(b)_{\mu}$ is a dense linear manifold in $\mathscr{H}(b)$. When $b$ is a non-extreme point of $H^\infty_1$, then $\mathscr{H}(b)_{\mu}$ contains $\mathscr{H}(b) \cap \mathcal{C}(\D^{-})$ which also turns out to be dense (see Section 2 below). This motivates our definition of reverse Carleson measure.

\begin{Definition}
For $\mu \in M_{+}(\D^{-})$ and $b \in H^{\infty}_{1}$ we say that $\mu$ is a \emph{reverse Carleson measure for} $\mathscr{H}(b)$ if $\HH(b)_{\mu}$ is dense in $\HH(b)$ 
and $\|f\|_{b} \lesssim \|f\|_{\mu}$ for all $f \in \mathscr{H}(b)_{\mu}$. 
\end{Definition}

In this definition, we allow the possibility for $\|f\|_{\mu}$ to be infinite. 

We are now ready to state our main reverse Carleson result. For an open arc $I$ in $\T$, let 
\begin{equation} \label{window}
S(I):=\left\{ z \in \D^{-} :\;\frac{z}{|z|}\in I,\;1-|z|\leq\frac{m(I)}{2}\right\}
\end{equation}
be the \emph{Carleson window over $I$}.

\begin{Theorem}\label{MainThmIntro}
Let $\mu \in M_{+}(\D^{-})$ and let $b$ be a non-extreme point of $H^{\infty}_{1}$ and $\mu$-admissible. If $h=d\mu|\T/dm$, then the following assertions are equivalent: 
\begin{enumerate}
\item The measure $\mu$ is a reverse Carleson mesure for $\mathscr{H}(b)$;
\item The inequality $\|k^{b}_\lambda\|_b \lesssim \|k^{b}_{\lambda}\|_{\mu}$ holds for every $\lambda\in\D$;
\item The measure $\nu$ defined by $d\nu:=(1 - |b|)d\mu$, 
%where $a$ is given by \eqref{def of a},  
satisfies
\begin{equation*}
%\label{revCar}
\inf_{I}\frac{\nu\left(S(I)\right)}{m(I)}>0;
\end{equation*}
\item We have $\mathrm{ess}\inf_{\T} (1 - |b|) h>0$. 
\end{enumerate}
\end{Theorem}

Let us place this theorem in some context. As we have already mentioned, when $\|b\|_\infty<1$, then $\mathscr{H}(b) = H^2$, with a norm equivalent to the usual $H^2$ norm.
In this situation, Lef\`evre {\it et al.} \cite{Queffelec} characterized reverse Carleson measures
under
the additional assumption that $\mu$ is already a Carleson measure. In  \cite{HMNOC},  the
authors were able to get rid of this extra assumption using a balayage type argument. This argument will play an important role in the proof of Theorem \ref{MainThmIntro}.
% and we recover a reverse Carleson result from \cite{HMNOC, Queffelec}. 

We should also point out that the reverse Carleson inequality stated in \cite{HMNOC} is tested on  the dense set $H^2 \cap \mathcal{C}(\D^{-})$, where $\mathcal{C}(\D^{-})$ denotes the complex-valued continuous functions on $\D^{-}$. 

 Of course the reader will immediately recognize that when the {\em infimum being positive} in statement (3) is replaced by the {\em supremum being finite}, we get the well known Carleson embedding condition which characterizes the boundedness of the embedding of $H^2$ into $L^2(\mu)$ (see \cite{Garnett}). 
 
 The implication (2) $\Rightarrow$ (1) is known as the (reverse) \emph{reproducing kernel thesis} and often appears in many Carleson  and reverse measure problems \cite{MR1864396, MR2417425}. When $b$ is an inner function, then reverse Carleson measures for model spaces were discussed in our recent paper \cite{BFGHR}. As it turns out there is no (reverse) reproducing kernel thesis in this setting \cite{HMNOC}. Reverse Carleson measures for other classical spaces were discussed in \cite{chacon, Luecking85,Luecking88}.

For a general, possibly extreme, point $b$ of $H^{\infty}_{1}$, we will prove the following. 

\begin{Theorem}
Suppose $\mu \in M_{+}(\D^{-})$ and $b \in H^{\infty}_{1}$ is $\mu$-admissible. If $h = d \mu|\T/dm$ and $\mu$ is a reverse Carleson measure for $\mathscr{H}(b)$ then 
\begin{equation}\label{NRE1}
(1 - |b|^2) \lesssim (1 - |b|^2)^{2} h
\end{equation}
$m$-almost everywhere on $\T$. 
\end{Theorem}

When $b$ is inner, the inequality in \eqref{NRE1} is trivial, while in other cases, as we will see
now,
it yields very important information:

\begin{Corollary}
Suppose $\mu \in M_{+}(\D^{-})$, $b \in H^{\infty}_{1}$ is $\mu$-admissible and not inner, $h = d \mu|\T/dm$, and $Z_{b} := \{\zeta \in \T: |b(\zeta)| < 1\}$. If $\mu$ is a reverse Carleson measure for $\HH(b)$ then 
$h \not \equiv 0$ and 
$$\int_{Z_b} \frac{1}{1 - |b|} dm < \infty.$$
\end{Corollary}
The above corollary says that any reverse Carleson measure for $\mathscr{H}(b)$, when $b$ is not inner, must have a non-zero absolutely continuous component with respect to $m$. In particular, there cannot be sampling sequences when $b$ is not inner (see Corollary \ref{Cor:sampling}). Notice how this is quite the dichotomy from the inner case where a reverse Carleson measure can be carried by $\D$ or even be singular with respect to $m$. 

We will also discuss (direct) Carleson measures for $\mathscr{H}(b)$. Here we make the following definition. 

\begin{Definition}
 A measure $\mu \in M_{+}(\D^{-})$ is a \emph{Carleson measure} for $\mathscr{H}(b)$ if $\mathscr{H}(b)_{\mu} = \mathscr{H}(b)$ and $\|f\|_{\mu} \lesssim \|f\|_{b}$ for all $f \in \mathscr{H}(b)$.
 \end{Definition}
 
 A result of Aleksandrov \cite{Alek} shows that when $b$ is inner then $\mathscr{H}(b)$ (which is just the model space $(b H^2)^{\perp}$) contains a dense set of continuous functions. Furthermore, if the embedding $\|f\|_{b} \lesssim \|f\|_{b}$ holds for the continuous functions in $\mathscr{H}(b)$ then {\em every} function in $\mathscr{H}(b)$ is $\mu$-admissible, i.e., $\mathscr{H}(b)_{\mu} = \mathscr{H}(b)$. Moreover, the embedding extends to all functions in $\mathscr{H}(b)$. 
 
 For a non-extreme point $b$ of $H^{\infty}_{1}$ there is a unique outer function $a$ with $a(0) > 0$ and such that $|a|^2 + |b|^2 = 1$ $m$-almost everywhere on $\T$. See Section 2 for more on this. 
Here is a sample result concerning Carleson measures. %theorem is the following. 
 
 \begin{Theorem}
Let $b$ be a rational and non-extreme point of $H^{\infty}_{1}$  and let $\mu \in M_{+}(\D^{-})$. Then the following assertions are equivalent: 
\begin{enumerate}
\item The measure $\mu$ is a Carleson measure for $\HH(b)$;
\item The measure $|a|^2\,d\mu$ is a Carleson measure for $H^2$. 
\end{enumerate}
\end{Theorem}

The reader might look at the definition of the measure $|a|^2 d \mu$ with some suspicion. However, when $b$ is rational then so is $a$ \cite[Remark 3.2]{MR2418122} and so $|a|^2 d \mu$ is clearly defined even if $\mu|\T$ has a non-trivial singular part with respect to $m$ (a point mass for example). Note that not every rational function in $H_1^\infty$ is non-extreme. For example, a finite Blaschke product is rational and extreme. See Section 5, where we consider more general (not only rational) functions $b$.  

When $b$ is inner then $\mathscr{H}(b)$ is a model space $(b H^2)^{\perp}$ and Carleson measures for these spaces were discussed in \cite{Al98, Baranov-JFA05, Cohn, Volberg-81, TV} (see also \cite{Kacnelson,Logvinenko,Panejah62,Panejah66} for some earlier related results). 

We will also examine when $L^2(\mu)$ can be used to define an equivalent norm of $\mathscr{H}(b)$, i.e., when each $f \in \mathscr{H}(b)$ is $\mu$-admissible and $\|f\|_{b} \asymp \|f\|_{\mu}$. This does indeed occur but only under very stringent circumstances. If we were to require the stronger condition that $\mu$ is an isometric measure, i.e., $\|f\|_{b} = \|f\|_{\mu}$ for all $f \in \mathscr{H}(b)$, then this never occurs:

\begin{Theorem}
If $b$ is non-constant and a non-extreme point of $H^{\infty}_{1}$, then there are no positive isometric measures for $\mathscr{H}(b)$.
\end{Theorem}

When $b$ is inner, there are plenty of isometric measures \cite{Al98, BFGHR}.

\section{Some reminders about $\HH(b)$ spaces}\label{section2}

For a wonderful detailed treatment of de Branges-Rovnyak spaces, we refer the reader to Sarason's book \cite{Sa} which contains essentially all the material presented in this section. Here, we merely set the notation and remind the reader of some standard facts we will use in this paper. 
For $\phi \in L^{\infty} := L^{\infty}(\T, m)$, define the Toeplitz operator on the classical Hardy space $H^2$ by 
$$T_{\phi} f = P_+(\phi f), \quad f \in H^2,$$
where $P_+$ is the orthogonal projection (often called the Riesz projection) of $L^2 := L^2(\T, m)$ onto $H^2$. 
Note, as is standard, how we regard $H^2$ as both a Hilbert space of analytic functions on $\D$ and, via non-tangential boundary values and Fourier coefficients, a closed subspace of $L^2$ \cite{Duren, Garnett}. We will use 
$$\langle f, g \rangle_2 := \int_{\T} f(\zeta) \overline{g(\zeta)} dm(\zeta)$$ for the inner product on $H^2$ (or $L^2$) and 
$\|f\|_{2} = \sqrt{\langle f, f\rangle_{2}}$ to denote the norm.
Also note that when $\phi \in H^{\infty}$ (the bounded analytic functions on $\D$), we have $T_{\phi} f = \phi f$ which is just a multiplication operator on $H^2$.

For $b \in H^{\infty}_1$ the \emph{de Branges-Rovnyak space} $\HH(b)$ is defined to be 
$$\HH(b) := (I - T_{b} T_{\overline{b}})^{1/2} H^2,$$ endowed with the inner product
$$\langle  (I - T_{b} T_{\overline{b}})^{1/2} f,  (I - T_{b} T_{\overline{b}})^{1/2} g\rangle_{b} := \langle f, g \rangle_2,$$
for $f, g \perp \ker ( (I - T_{b} T_{\overline{b}})^{1/2})$. That is to say, $\HH(b)$ is normed to make $(I - T_{b} T_{\overline{b}})^{1/2}$ a partial isometry of $H^2$ onto $\HH(b)$. When $\|b\|_{\infty} < 1$, the operator $I-T_bT_{\bar b}$ is an isomorphism on $H^2$ and thus $\HH(b) = H^2$ with an equivalent norm. On the other extreme, when $b$ is an inner function then $T_bT_{\bar b}$ is the orthogonal projection of $H^2$ onto $bH^2$ and thus $\HH(b)$ turns out to be $(b H^2)^\perp=H^2\ominus bH^2$ with the standard $H^2$ norm.

The space $\HH(b)$ is a reproducing kernel space with kernel 
$$k^{b}_{\lambda}(z) := \frac{1 - \overline{b(\lambda)} b(z) }{1 -  \overline{\lambda} z}, \quad \lambda,z\in\D, 
$$ i.e., 
$$f(\lambda) = \langle f, k^{b}_{\lambda} \rangle_{b}, \quad \lambda \in \D, f \in \HH(b).$$ We point out that if 
\begin{equation} \label{Cauchykernel}
k_{\lambda}(z) := \frac{1}{1 - \overline{\lambda} z}
\end{equation}
 is the standard reproducing kernel for $H^2$, then 
$$k^{b}_{\lambda} = (I - T_{b} T_{\overline{b}}) k_{\lambda}.$$
Observe the notation here: $k_\lambda^b$ 
% (notice the $b$ in the superscript) 
is the reproducing kernel for $\HH(b)$ while 
$k_\lambda$ %(without the $b$ in the superscript) 
is the reproducing kernel for $H^2$. 

%We mention that one could also define $\mathscr{H}(b)$ by first populating it with the kernels $k_{\lambda}^{b}$ and then using the fact that this kernel is positive definite to define $\mathscr{H}(b)$ as the corresponding reproducing kernel Hilbert space in the standard way.

As already mentioned in the introduction, starting with the positive definite kernel $k^{b}_{\lambda}$,  $\HH(b)$ can also be defined as the reproducing kernel Hilbert space associated with this kernel.

We will now assume for the main part of the paper that $b$ is a \emph{non-extreme} point of the unit ball of $H^\infty$, which,  by  
the Arens--Buck--Carleson--Hoffman--Royden \cite{MR0098981} theorem, is equivalent to the condition
\begin{equation}\label{extreme-defn} 
-\infty<\int_{\T} \log (1 - |b|) dm.
\end{equation}
To abbreviate, we will simply say $b$ is non-extreme. In this case there is a unique outer function $a$ with $a(0) > 0$ such that 
\begin{equation}\label{def of a}
|a(\zeta)|^2 + |b(\zeta)|^2 = 1 \quad m-\mbox{a.e.} \; \; \zeta \in \T.
\end{equation}
We call a pair $(a, b)$ satisfying the above equality a\emph{ Pythagorean pair}. 

For $\phi \in L^{\infty}$, let 
$$\mathscr{M}(\phi) := T_{\phi} H^2$$
endowed with the norm which makes $T_{\phi}$ a partial isometry from $H^2$ onto $T_{\phi} H^2$. Observe that when $a\in H^{\infty}$ and is outer, then $T_a$ and $T_{\ov{a}}$
are one-to-one and hence 
\[
\|T_af\|_{\mathscr{M}(a)}=\|T_{\ov{a}}f\|_{\mathscr{M}(\ov{a})}= \|f\|_2,\quad f\in H^2.
\]
When $b$ is non-extreme  then $\mathscr{M}(a) = a H^2$ is contractively contained in $\mathscr{M}(\overline{a})$ which is, in turn, contractively contained in $\HH(b)$. Moreover, $\mathscr{M}(\overline{a})$ is dense in $\HH(b)$ and the linear span of the reproducing kernels $k_\lambda$ (for $H^2$), $\lambda\in\D$, is contained and dense in $\mathscr{M}(\overline{a})$ -- and thus it is also dense in $\HH(b)$. In particular, we see that $\HH(b)\cap\mathcal{C}(\D^{-})$ is dense in $\HH(b)$. 

It is also known, when $b$ is non-extreme, that for every $f\in\HH(b)$,
we have $T_{\bar b}f\in \mathscr{M}(\overline{a})$ and one can obtain the norm of $f$ via the formula
\begin{equation}
\label{normdBR}
\Vert f\Vert_{b}^2=\Vert f\Vert_{2}^2+\Vert g\Vert_{2}^2,
\end{equation}
where $g$ is defined by 
\begin{equation} \label{TbTa}
T_{\ov{b}}f=T_{\ov{a}}g.
\end{equation}
Note  that $g$ is unique since, as discussed earlier, $T_{\overline a}$ is one-to-one due to the fact that $a$ is outer. At least when $(\overline{b/a})f\in L^2$, it can be checked that
\begin{equation}\label{Tba}
g = T_{\overline{b/a}}f.
\end{equation}

We say that $(a, b)$ forms a {\it corona pair} if
$$\inf\{|a(z)| + |b(z)|: z \in \D\} > 0.$$
Still under the hypothesis that $b$ is non-extreme, we have that $\mathscr{M}(\overline{a}) = \HH(b)$ (with equivalent norms) if and only if $(a, b)$ forms a corona pair. We also have $\mathscr{M}(a) = \HH(b)$ (with equivalent norms) if and only if $(a, b)$ forms a corona pair and $T_{a/\overline{a}}$ is invertible on $H^2$. 

Recall (see e.g.\ \cite[Theorem B4.3.1]{Nik}) that $T_{a/\bar a}$ is invertible if and only if $|a|^2$ satisfies the \emph{Muckenhoupt $(A_2)$ condition}, i.e.,
\begin{equation} \label{MC}
\sup_{I} \left(\frac{1}{m(I)}\int_{I}|a|^{-2}\,dm\right)\left(\frac{1}{m(I)}\int_{I}|a|^2\,dm\right)<\infty,
\end{equation}
where $I$ runs over all  subarcs of $\T$. 
For shorthand  we will often write 
$$|a|^2 \in (A_2)$$
when $|a|^2$ satisfies \eqref{MC}.  
The $(A_2)$ condition is equivalent to the boundedness of  the Riesz projection $P_+$  from $L^2(|a|^2\,dm)$ to itself
(or from $L^2(|a|^{-2}\,dm)$ to itself). 

We end this section by noting that if $(a,b)$ is a Pythagorean pair, the $\mu$-admissibility of one function does not imply that of the second one. Indeed pick $\mu=\delta_1$, the Dirac measure at the point $1$. Let $a_0$ be an outer function bounded by $1$ which has no radial limit at $1$. Multiply $a_0$ by the singular inner function $I$, defined by $I(z)=\exp((z+1)/(z-1))$. Then $a=a_0 I$ has a radial limit $0$ at $1$. Now $a$ and $a_0$ have the same Pythagorean mates $e^{i\theta}b$
 ($\theta\in\R$). 
%Thus we see that $a$ is $\mu$-admissible whereas $b$ is not. 
Then either $b$ has radial limit at $1$ and $a_0$ has not, or $b$ has no radial limit at $1$ but
$a$ has.

\section{A first observation about reverse Carleson measures}\label{section3}

When $b$ is  non-extreme there is the following interesting relationship between the reverse Carleson measure condition and the Pythagorean pair $(a, b)$. Recall from the previous section that since $b$ is  non-extreme,  the reproducing kernels $k_\lambda$ for $H^2$  belong to $\HH(b)$.

\begin{Proposition} \label{bainh2}
\label{b/a}
If $b\in H^\infty_1$ is  non-extreme and
 $\mu \in M_{+}(\D^{-})$ satisfies 
 $$\Vert k_\lambda \Vert_{b} \lesssim \Vert k_\lambda \Vert_{\mu} \quad \lambda\in\D, $$
then $b/a \in H^2$.
\end{Proposition}

\begin{proof}
Since
\[
T_{\bar b}k_\lambda=\overline{b(\lambda)} k_\lambda=\overline{b(\lambda)}\,{\overline{a(\lambda)}}^{-1}T_{\bar a}k_\lambda,
\]
we get from \eqref{normdBR}
\begin{equation}\label{eq:norm-Hb-kernel}
\|k_\lambda\|_b^2=\|k_\lambda\|_2^2+\frac{|b(\lambda)|^2}{|a(\lambda)|^2}\|k_\lambda\|_2^2=\left(1+\frac{|b(\lambda)|^2}{|a(\lambda)|^2}\right)\frac{1}{1-|\lambda|^2}.
\end{equation}
Hence there is a constant $C>0$ such that 
\[
\frac{|b(\lambda)|^2}{|a(\lambda)|^2}\leq C \int_{\D^{-}}\frac{1-|\lambda|^2}{|1-\bar\lambda z|^2}\,d\mu(z),\quad \lambda\in\D.
\]
Setting $\lambda=re^{it}$, integrating both sides of the previous inequality over $t\in (0,2\pi)$, and using Fubini's theorem, we obtain 
\[
\frac{1}{2\pi}\int_0^{2\pi}\frac{|b(re^{it})|^2}{|a(re^{it})|^2}\,dt\leq C \int_{\D^{-}}\left(\int_0^{2\pi}\frac{1-r^2}{|1-rze^{-it}|^2}\,\frac{dt}{2\pi}\right)\,d\mu(z).
\]
Using basic properties of the Poisson kernel we get
\[
\int_0^{2\pi}\frac{1-r^2}{|1-rze^{-it}|^2}\,\frac{dt}{2\pi}=\frac{1-r^2}{1-r^2|z|^2}\leq 1,
\]
for every $z\in\D^{-}$, which yields 
\[
\frac{1}{2\pi}\int_0^{2\pi}\frac{|b(re^{it})|^2}{|a(re^{it})|^2}\,dt\leq C \mu(\D^{-})< \infty
\]
for every $0<r<1$. Hence (by the definition of the Hardy space) $b/a \in H^2$.
%Since
%\[
%\Vert z^{n}\Vert_{b} \lesssim \Vert z^{n}\Vert_{\mu}\leq \mu(\D^{-})<\infty,
%\]
%we have
%\[
%\sup_{n\in\N}\Vert z^{n}\Vert_{b}<\infty.
%\]
%To compute $\Vert z^n\Vert_{b}$, we will use \eqref{normdBR} and so we need to calculate $\Vert h_n\Vert_{2}$, where $$T_{\ov{b}}z^n=T_{\ov{a}}h_{n}.$$ Since $a$ is outer then $b/a$ is analytic in $\D$ and so we can expand $b/a$ as a power series as 
%\[
%\frac{b}{a}(z)=\sum_{n=0}^{+\infty}c_{n}z^{n}.
%\]
%Hence, the function $\widetilde{h}_n$ whose Fourier expansion is 
%\[
%\widetilde{h}_{n}(\zeta):=\ov{\zeta}^{n}\left(b(\zeta)-a(\zeta)\sum_{j=0}^{n}c_{j}\zeta^{j}\right)
%\]
%belongs to $H_{0}^2$ (Hardy space functions which vanish at the origin). We observe that (in terms of Fourier series)
%\[
%\zeta^n\ov{b(\zeta)}=\ov{\widetilde{h}_{n}(\zeta)}+\ov{a(\zeta)}\sum_{j=0}^{n}\ov{c_{j}}\zeta^{n-j}.
%\]
%Thus, remembering that $P$ is the orthogonal projection of $L^2$ onto $H^2$,
%\[
%P(\zeta^{n}\ov{b(\zeta)}) = \left(\overline{a(\zeta)}\sum_{j=0}^{n}\ov{c_{j}}\zeta^{n-j}\right)
%\]
%which, by \eqref{TbTa},  implies that $h_n=\sum_{j=0}^{n}\ov{c_{j}}z^{n-j}$ and so by \eqref{normdBR}
%\[
%\Vert z^n\Vert_{b}^2=1+\sum_{j=1}^{n}|c_j|^2.
%\]
%We finally obtain that 
%$$\sup_{n\in\N}\left(\sum_{j=1}^{n}|c_j|^2\right)<\infty$$ which amounts to saying that the sequence $\{c_n\}_{n \geq 0}$ is square summable, or equivalently,  $b/a\in H^2$.
\end{proof}

We can connect Proposition \ref{b/a} to a series of other well-known equivalences \cite{MR847333}.

\begin{Theorem}[Sarason]\label{Sarasonequivalence}
%We mention that, from a result of Sarason, we know that the following conditions are equivalent:
Let $(a,b)$ be a Pythagorean pair. Then the following assertions are equivalent:
\begin{enumerate}
\item The function $b/a$ belongs to $H^2$;  %\Longleftrightarrow 
\item The space $H^\infty$ is contained in $\mathscr H(b)$; %\Longleftrightarrow
\item We have $\sup_{n\in\N}\Vert z^{n}\Vert_{b}<\infty$; % \Longleftrightarrow
\item The function $\left(1-|b|\right)^{-1}$ belongs to $L^{1}$.
\end{enumerate}
\end{Theorem}

\begin{Remark}\label{rem-on-b-over-a}
 If $b/a$ satisfies the stronger condition $b/a \in H^{\infty}$, then $\Vert b\Vert_{\infty}<1$ and so $\HH(b)=H^2$ (with equivalent norms). To see this, write $b=ah$, for some $h\in H^{\infty}$. It follows that 
\[
1=|a|^2+|b|^2=|a|^{2}\left(1+|h|^{2}\right)\quad \text{a.e.\ on }\T.
\]
From here we see that $1/a\in L^{\infty}$ which implies $\|b\|_{\infty} < 1$.
\end{Remark}

The next result says that not every $\mathscr{H}(b)$ space admits a reverse Carleson measure. 

\begin{Corollary}\label{no-RCE-non-ext}
Let $b \in H^{\infty}_{1}$ be non-extreme. Then $\mathscr{H}(b)$ admits a reverse Carleson measure 
if and only if $(1 - |b|)^{-1} \in L^1$. Thus there are $\HH(b)$ spaces with non-extreme $b$  which do not admit reverse Carleson measures. 
\end{Corollary}

\begin{proof}
Suppose that $\HH(b)$ admits a reverse Carleson measure. Then, by Proposition \ref{bainh2}, along with Theorem \ref{Sarasonequivalence}, we see that $(1 - |b|)^{-1} \in L^1$. The converse will follow from Theorem \ref{MainThm} below (see Remark \ref{R-sorted}). 
\end{proof}

A specific example of an $\mathscr{H}(b)$ space which admits no reverse Carleson measures is with 
$b(z) = (1 - z)/2$. This is because $1/(1 - |b|) \not \in L^1$. 

\section{The Main reverse Carleson measure result}\label{section4}

For $\mu \in M_{+}(\D^{-})$, recall that $b$ is $\mu$-admissible if the non-tangential limits of $b$ exist $\mu$-almost everywhere on $\T$.  Also recall that when $b$ is $\mu$-admissible then $\mathscr{H}(b)_{\mu}$, the set of all $\mu$-admissible functions in $\mathscr{H}(b)$, contains the reproducing kernels $k_\lambda^b$, $\lambda\in\D$, and thus is a dense linear manifold in $\mathscr{H}(b)$. Here is our main reverse Carleson measure result. 

\begin{Theorem}\label{MainThm}
Let $\mu \in M_{+}(\D^{-})$, $b\in H^\infty_1$ be non-extreme and $\mu$-admissible, and let $h=d\mu|\T/dm$. Then the following assertions are equivalent: 
\begin{enumerate}
\item The measure $\mu$ is a reverse Carleson measure for $\mathscr{H}(b)$;
\item The inequality $\|k^{b}_\lambda\|_b\lesssim \|k^{b}_\lambda\|_\mu$ holds for every $\lambda\in\D$;
\item The measure $\nu$ defined by $d\nu:=(1 - |b|^2)d\mu$ satisfies the condition
\begin{equation}
\label{revCar}
\inf_{I}\frac{\nu\left(S(I)\right)}{m(I)}>0;
\end{equation}
\item We have $\mathrm{ess}\inf_{\T} (1 - |b|^2)h >0$. 
\end{enumerate}
\end{Theorem}

\begin{Remark}\label{R-mu-nu}
\begin{enumerate}
\item Notice that since statement (2) of the above theorem implies statement (1), the reverse reproducing kernel thesis holds for $\HH(b)$ when $b$ is  non-extreme.
This is not necessarily the case when $b$ is extreme.  For example, in \cite{HMNOC} it is shown that 
%the reverse reproducing kernel thesis can fail 
whenever $b$ is an inner function, then there is a measure satisfying (2) but not (1).
\item The part of the measure guaranteeing reverse Carleson embeddings in $\HH(b)$ is supported on $\T$ (with the control given in statement (4) of theorem). For example, any measure carried only by $\D$ can not be a reverse Carleson measure. 
\end{enumerate}
\end{Remark}

The proof of Theorem \ref{MainThm} will require this technical lemma from harmonic analysis. This is surely a `folk theorem' but we prove it here for the reader's convenience. 

\begin{Lemma}\label{HA}
Let $q$ be a bounded analytic function on $\D$.
Then for almost every $\zeta \in \T$, 
\begin{equation} \label{q-vqr}
\lim_{r \to 1^{-}} \int_{\T}|q(r\xi)|^{2}\frac{1-r^{2}}{|\xi-r \zeta|^{2}}dm(\xi) =|q(\zeta)|^2.
\end{equation}
\end{Lemma}

\begin{proof}
Let us suppose that $\zeta \in \T$ is a Lebesgue point of $|q|^2$ where $q$ admits a radial limit $l$.
Let $u$ be the harmonic function on $\D$ whose (almost everywhere defined) radial limits $u(\xi)$, $\xi \in \T$, satisfy $u(\xi)=|q(\xi)|^2$ a.e. Then,
by the fact that $|q|^2$ is subharmonic, 
we have $|q|^{2}\leq u$ on $\D$ and it follows that
\begin{eqnarray*}
\int_{\T}|q(r\xi)|^{2}\frac{1-r^{2}}{|\xi-r \zeta|^{2}}dm(\xi) 
 & \leq & \int_{\T}u(r\xi)\frac{1-r^{2}}{|\xi-r \zeta|^{2}}dm(\xi) =u(r^{2} \zeta).
\end{eqnarray*}

By our assumptions on $\zeta$, the above implies that 
\begin{equation}
\label{lim sup}
0\le \underset{r \to 1}{\ov{\lim}} \int_{\T}|q(r\xi)|^{2}\frac{1-r^{2}}{|\xi -r \zeta|^{2}}dm(\xi) \leq |q(\zeta)|^{2}=|l|^2, %\quad \text{a.e. } z\in I.
\end{equation}
and this is true almost everywhere.
%Since it is true almost everywhere, we can assume that $|a|$ has radial limit $l$ at $z$ and 
We now set $\widetilde{q}:=q-l$.
Repeating the above argument for $\widetilde{q}$,
from (\ref{lim sup}), we deduce that 
\begin{eqnarray*}
 \underset{r \to 1}{\ov{\lim}} \int_{\T}|q(r\xi)-l|^{2}\frac{1-r^{2}}{|\xi -r \zeta|^{2}}dm(\xi) &=&\underset{r \to 1}{\ov{\lim}} \int_{\T}|\widetilde{q}(r\xi)|^{2}\frac{1-r^{2}}{|\xi -r \zeta|^{2}}dm(\xi) \\ &\leq& |\widetilde{q}(\zeta)|^{2}=0,
\end{eqnarray*}
which, since the expression on the left hand side is always non-negative, 
allows us to switch from a lim sup to a regular limit, i.e., 
\[
\underset{r \to 1}{{\lim}} \int_{\T}|q(r\xi)-l|^{2}\frac{1-r^{2}}{|\xi -r \zeta|^{2}}dm(\xi) =0.
\]

On the other hand,
\begin{eqnarray*}
\lefteqn{\int_{\T}|q(r\xi)-l|^{2}\frac{1-r^{2}}{|\xi-r \zeta|^{2}}dm(\xi)}\\
 && =  \int_{\T}|q(r\xi)|^{2}\frac{1-r^{2}}{|\xi-r \zeta|^{2}}dm(\xi)+|l|^{2}
  -  2\int_{\T}\Re(\ov{l}q(r\xi))\frac{1-r^{2}}{|\xi-r \zeta|^{2}}dm(\xi)\\
 && =  \int_{\T}|q(r\xi)|^{2}\frac{1-r^{2}}{|\xi-r \zeta|^{2}}dm(\xi)+|l|^{2}-2
\Re(\ov{l}q(r^{2}\zeta))
\end{eqnarray*}
(where we have used the fact that  $\Re(\bar l q_r)$ is harmonic, with $q_r(\zeta)=q(r\zeta)$)
and so
\[
\underset{r\to 1}{\lim} \int_{\T}|q(r\xi)|^{2}\frac{1-r^{2}}{|\xi-r \zeta|^{2}}dm(\xi)=|l|^2 =|q(\zeta)|^2,
\]
which is the desired conclusion.
\end{proof}

\begin{Remark} \label{r-q-vqr-I}
%\begin{enumerate}
%\item The previous lemma obviously holds when $|q|^2$ is replaced by $q$ in \eqref{q-vqr}. 
%\item 
It follows from Lemma~\ref{HA} and basic facts about the Poisson kernel that, for an interval $I$,
$$
\lim_{r \to 1^{-}} \int_{I}|q(r\xi)|^{2}\frac{1-r^{2}}{|\xi-r \zeta|^{2}}dm(\xi)$$
is equal almost everywhere to $|q(\zeta)|^2$ when $\zeta$ is in the interior of $I$ and zero when $\zeta$ does not belong to the closure of $I$.
%\end{enumerate}
\end{Remark}

\begin{proof}[Proof of Theorem \ref{MainThm}]
The structure of the proof will be to show: $(1) \Rightarrow (2) \Rightarrow (4) \Leftrightarrow (3)$ and $(4) \Rightarrow (1)$.

Statement $(1)$ implies $(2)$ is clear since $b$ is $\mu$-admissible and so  $k^{b}_{\lambda} \in \mathscr{H}(b)_\mu$ for every $\lambda \in \D$. 
Statement $(3)$ implies $(4)$ follows from the Lebesgue differentiation theorem and the fact that the (symmetric) derivative of a singular measure is zero $m$-almost everywhere.  

For the implication $(4)\Rightarrow (3)$, it is sufficient to note that, for any open arc $I$ of $\mathbb T$, we have
$$
\nu(S(I))=\int_{S(I)}(1 - |b|^2)\,d\mu\geq \int_I (1 - |b|^2) h\,dm\geq \delta m(I),
$$
where $$\delta := \mathrm{ess}\inf_{\T} (1 -  |b|^2) h >  0.$$

Let us prove the implication $(4)\Rightarrow (1)$. Let $\delta > 0$ be defined as in the previous line. To test the reverse Carleson condition, we just need to test it on $f \in \mathscr{H}(b)_{\mu}$ for which $\|f\|_{\mu} < \infty$. 
For such functions $f$ we have, with $a$ being the Pythagorean mate for $b$, 
\[
\int_\T |f|^2 |a|^{-2}\,dm\leq \delta^{-1}\int_\T |f|^2 h\,dm=\delta^{-1} \int |f|^2\,d\mu < \infty.
\]
Thus $f/a \in L^2$. But since $a$ is outer, this yields (via a standard fact from Hardy spaces) that $f/a\in H^2$ and thus $f=af/a\in \mathscr{M}(a)$.   Now using the fact that $\mathscr{M}(a)$ is contractively contained in $\HH(b)$, we get 
\[
\|f\|_b\leq \|f\|_{\mathscr{M}(a)}=\|f/a\|_2\leq \delta^{-1/2} \|f\|_\mu
\]
which is the desired inequality. 

It remains to check the implication $(2) \Rightarrow (4)$. If $\|k^{b}_{\lambda}\|_{b} \lesssim \|k^{b}_{\lambda}\|_{\mu}$ for all $\lambda\in\D$ then 
$$\frac{1 - |b(\lambda)|^2}{1 - |\lambda|^2} \lesssim \int_{\D^{-}} \left| \frac{1 - \overline{b(\lambda)} b(z)}{1 - \overline{\lambda} z}\right|^2 d \mu(z).$$ This says that 
\begin{eqnarray}\label{qqq}
1 - |b(\lambda)|^2 & \lesssim {\displaystyle \int_\D  \frac{1 - |\lambda|^2}{|1 - \overline{\lambda} z|^2} |1 - \overline{b(\lambda)} b(z)|^2 d \eta(z)} \\ \nonumber
& + {\displaystyle  \int_\T \frac{1 - |\lambda|^2}{|1 - \overline{\lambda} z|^2} |1 - \overline{b(\lambda)} b(z)|^2  h(z) dm(z)}\\ \nonumber 
& + {\displaystyle  \int_\T  \frac{1 - |\lambda|^2}{|1 - \overline{\lambda} z|^2} |1 - \overline{b(\lambda)} b(z)|^2 d \sigma(z),}
\end{eqnarray}
where  
$$d \mu = d \eta + h dm + d \sigma,$$
	$\eta = \mu|\D$,  $h = d \mu|\T/dm$, and $\sigma$ is the singular part of $\mu|\T$. Note that the three integrals on the right-hand side of \eqref{qqq} are well defined since we are assuming that $b$ is a $\mu$-admissible, bounded analytic function. 
%Of course the first integral is defined since $b$ is analytic on $\D$ and $\eta$ is carried by $\D$ and the second is defined since the boundary values of $b$ are defined $m$-a.e. by standard Hardy spaces theory. 

Let $E$ be the measurable subset of $\zeta \in \T$ which satisfies the following conditions:
\begin{equation} \label{V1}
\mbox{$\zeta$ is a Lebesgue point of $h$;}
\end{equation}
\begin{equation} \label{V2}
\mbox{$b$ has a radial limit at $\zeta$;}
\end{equation}
\begin{equation} \label{V3}
\mbox{$(D \sigma)(\zeta) = 0$,}
\end{equation}
where $D \sigma$ is the symmetric derivative of $\sigma$ -- which is zero $m$-almost everywhere. 
 Notice that $E$ is a set of full Lebesgue measure in $\T$. 

Let $I$ be a sub-arc of $\T$ containing $\zeta \in E$. For this interval define 
$$S(I, y) := \{z \in S(I): |z| \geq 1 - y\}.$$

Integrating the left-hand side of \eqref{qqq} over $S(I, y)$ and dividing by $y$ we get 
$$\frac{1}{y} \int_{1 - y}^{1} \left( \int_{I} (1 - |b(r e^{i t})|^2) dt\right) r dr.$$ One can argue by the dominated convergence theorem that this quantity goes to 
$$\int_{I} (1 - |b(e^{i t})|^2) dt \quad \mbox{as $y \to 0$}.$$ 

Now integrate the first integral on the right-hand side of \eqref{qqq} over $S(I, y)$ and divide by $y$ to get (after applying Fubini's theorem)
\begin{equation} \label{first-int}
\int_{\D} \frac{1}{y} \int_{1 - y}^{1}\left( \int_{I} \frac{1 - r^2}{|1 - r e^{- i t} z|^2} |1 - \overline{b(r e^{i t})} b(z)|^2  dt \right) r dr d\eta(z).
\end{equation}
Note that the inner two integrals, i.e., 
$$ \frac{1}{y} \int_{1 - y}^{1}\left( \int_{I} \frac{1 - r^2}{|1 - r e^{- i t} z|^2} |1 - \overline{b(r e^{i t})} b(z)|^2 dt \right) r dr$$ is bounded above by a constant times 
$$\frac{1}{y} \int_{1 - y}^{1} \left(\int_{I} \frac{1 - r^2}{|1 - r e^{- i t} z|^2} dt \right) r dr$$ which approaches $\chi_{I}$ as $y \to 0$. 
Thus the integral in \eqref{first-int} is bounded above by a quantity which approaches 
\begin{equation} \label{muD}
\int_{\D} \chi_{I} d \eta,\quad \mbox{as $y \to 0$},
\end{equation}
which is equal to zero 
since $I \cap \D=\emptyset$.
%\subset \T$ and the integral is over $\D$ (the \emph{open} unit disk). 

Now integrate the second integral on the right-hand side of \eqref{qqq} over $S(I, y)$ and divide by $y$ to get (after applying Fubini's theorem)
$$\int_{\T} \frac{1}{y} \int_{1 - y}^{1} \left(\int_{I} \frac{1 - r^2}{|1 - e^{-i t} r \zeta|^2} |1 - \overline{b(r e^{i t})} b(\zeta)|^2 dt \right) r dr h(\zeta) dm(\zeta).$$ 
Apply Lemma \ref{HA} and Remark \ref{r-q-vqr-I}  to
the inner integral 
$$\int_{I} \frac{1 - r^2}{|1 - e^{-i t} r \zeta|^2} |1 - \overline{b(r e^{i t})} b(\zeta)|^2 dt $$ 
%is equal to 
%\begin{align*}
%\int_{I} \frac{1 - r^2}{|1 - r \zeta e^{-i t}|^2} dt & - 2 \Re \left(\overline{b(\zeta)} \int_{I} \frac{1 - r^2}{|1 - r \zeta e^{-i t}|^2} b(r e^{i t}) dt \right)\\
%&+ |b(\zeta)|^2 \int_{I} \frac{1 - r^2}{|1 - r \zeta e^{-i t}|^2} |b(r e^{i t})|^2 dt.
%\end{align*}
%Apply Lemma \ref{HA} and Remark \ref{r-q-vqr-I} 
to see that this quantity approaches $(1 - |b(\zeta)|^2)^2 \chi_{I}$ as $r \to 1$. Thus the second integral on the right-hand side of \eqref{qqq} (over $S(I, y)$ and divide by $y$) approaches 
\begin{equation} \label{second-int-RHS}
\int_{I} (1 - |b|^2)^2 h dm,\quad\mbox{as $y \to 0$}.
\end{equation}

Now use the exact same proof used to get \eqref{muD} to show that the third integral on the right-hand side of \eqref{qqq} (over $S(I, y)$ and divide by $y$) approaches 
$$\int_{I} d\sigma,\quad\mbox{as $y \to 0$}.$$

Combining our results we get  
$$\int_{I} (1 - |b|^2) dm \lesssim \int_{I} (1 - |b|^2)^{2} h dm + \int_{I} d \sigma.$$
Now, %proceed to use the Lebesgue differentiation theorem (
remembering that $D\sigma$ is zero on $E$, we get the required result. 
\end{proof}

\begin{Remark} \label{R-sorted}
\begin{enumerate}
\item The above proof can be suitably modified to show that the family of Cauchy kernels $\{k_\lambda:\lambda\in\mathbb D\}$ can be used to test the reverse embedding. More precisely, if 
$\|k_\lambda\|_b\lesssim \|k_\lambda\|_\mu$ for every $\lambda$ in $\mathbb D$, then $\|f\|_b\lesssim \|f\|_\mu$ for every $f$ in $\mathscr H(b)_{\mu}$.  Since the kernels $k_{\lambda}$ are simpler than $k^{b}_{\lambda}$, this could, in certain circumstances, provide an easier test for reverse Carleson measures. 
\item Theorem \ref{MainThm} can be used to complete the proof of the converse of Corollary \ref{no-RCE-non-ext}. Indeed if $(1 - |b|)^{-1}$ belongs to $L^1$ then the measure $d\mu:=(1 - |b|)^{-1} dm$ is finite, $b$ is admissible with respect to this measure, and $\mu$ is a reverse Carleson measure for $\mathscr{H}(b)$ since $\mathrm{ess}\inf_{\T} (1 - |b|^2)(1-|b|)^{-1}\geq 1>0$.
\end{enumerate}
\end{Remark}

When $b\in H^\infty_1$ (not necessarily non-extreme), the proof of the implication
(2) $\Rightarrow$ (4) of Theorem \ref{MainThm} actually shows the following.

\begin{Theorem}\label{thm:reverse-cas-general}
Suppose $\mu \in M_{+}(\D^{-})$, $b \in H^{\infty}_{1}$ is $\mu$-admissible, and $h = d \mu|\T/dm$. If $\mu$ is a reverse Carleson measure for $\mathscr{H}(b)$ then
\begin{equation} \label{NRE}
(1 - |b|^2) \lesssim (1 - |b|^2)^{2} h.
\end{equation}
almost everywhere on $\T$. 
\end{Theorem}

When $b$ is inner then the inequality in \eqref{NRE}, though true, yields no information. When $b$ is non-extreme then the condition that the Lebesgue measure of the set where $|b| = 1$ is zero along with Condition~\ref{NRE} is equivalent to Condition (4) in Theorem~\ref{MainThm}.  We have the following general corollary. 
\begin{Corollary}\label{Cor:cn-reverse}
Suppose $\mu \in M_{+}(\D^{-})$, $b \in H^{\infty}_{1}$ is not inner and $\mu$-admissible. Let  $h = d \mu|\T/dm$ and $Z_{b} := \{\zeta \in \T: |b(\zeta)| < 1\}$.  Assume that $\mu$ is a reverse Carleson measure for $\mathscr{H}(b)$, then $h \not \equiv 0$ and 
 \[
\int_{Z_b}\frac{1}{1-|b|}\,dm<+\infty.
\]
Moreover, if $m(Z_b)=1$, then $b$ is non-extreme. 
\end{Corollary}
\begin{proof}
By Theorem~\ref{thm:reverse-cas-general}, the inequality \eqref{NRE} holds and since $b$ is not inner, this inequality implies that $h \not \equiv 0$. On $Z_b$ we now obtain from \eqref{NRE} that 
\[
1\lesssim (1-|b|^2)h,
\]
that is $(1-|b|)^{-1}\lesssim h$ a.e. on $Z_b$. Since $h\in L^1(\T)$, we see that $\int_{Z_b}(1-|b|)^{-1}\,dm$ is finite. If furthermore $m(Z_b)=1$, the integrability of $(1-|b|)^{-1}$ implies that of $\log(1-|b|)$ and so $b$ is non-extreme. 
 \end{proof}
 
Our results have an interesting connection to sampling sequences for $\HH(b)$ spaces. Recall that if $\HH$ is a reproducing kernel Hilbert space on a set $\Omega$, and if $k_\lambda^{\HH}$ denotes its reproducing kernel at point $\lambda$, then a sequence $(\lambda_n)_{n\geq 1}\subset\Omega$ is called a \emph{sampling sequence} for $\HH$ if 
\[
\|f\|_\HH^2\asymp \sum_{n=1}^{\infty}\|k_{\lambda_n}^\HH\|_\HH^{-2} |f(\lambda_n)|^2,
\]  
for all $f\in\HH$.
 \begin{Corollary}\label{Cor:sampling}
Let $b\in H_1^\infty$. If $\HH(b)$  admits a sampling sequence, then necessarily $b$ is an inner function.
\end{Corollary}

\begin{proof}
Assume that there exists a sequence $(\lambda_n)_{n\geq 1}\subset\D$ which is a sampling sequence for $\HH(b)$. This implies, in particular, that the measure 
\[
\mu:=\sum_{n=1}^{+\infty}\|k_{\lambda_n}^b\|_b^{-2}\delta_{\lambda_n}
\]
is a reverse Carleson measure for $\HH(b)$. But since $d \mu|\T/dm\equiv 0$, this contradicts Corollary~\ref{Cor:cn-reverse}. 
\end{proof}
This Corollary generalizes a result obtained in \cite{MR2514455}. Note also that the proof given above shows that $\HH(b)$ does not have an orthogonal basis of reproducing kernels if $b$ is not inner. This result was already proved in \cite{MR2159461} using a different method based on spectral perturbation and originally coming from Clark's theory.

\begin{Example}
Let $b$ be the outer function whose modulus satisfies 
$$|b(e^{i \theta})| = 1 - \exp(- \frac{1}{\theta^2}).$$
Then 
$$\int_{\T} \log \frac{1}{1  - |b|} dm \asymp \int_{0}^{2 \pi} \frac{1}{\theta^2}d \theta  = \infty$$
and so $b$ is  extreme. Moreover, $Z_{b} = \T \setminus \{1\}$. In particular, $m(Z_b)=1$ and by Corollary~\ref{Cor:cn-reverse}, $\mathscr{H}(b)$ will have no reverse Carleson measures. 
\end{Example}

Conspicuously missing from this discussion is the case where $b$ is extreme and not inner and for which 
$$\int_{Z_b} \frac{1}{1 - |b|} dm < \infty.$$ An example of this would be the outer function $b$ whose modulus is $1$ on $\T \cap \{\Im z > 0\}$ and $1/2$ on $\T \cap \{\Im z < 0\}$. Do such $\mathscr{H}(b)$ spaces have reverse Carleson measures? 

\section{An analogous condition for direct embeddings}\label{section5}

In this section, we again assume $b$ is non-extreme  and $a$ is defined by (\ref{def of a}). We also recall the definition of Carleson measure from the introduction ($\mathscr{H}(b)_{\mu}  = \mathscr{H}(b)$ and $\|f\|_{\mu} \lesssim \|f\|_{b}$ for every $f \in \mathscr{H}(b)$). We remind the reader that $\mathscr{M}(a)=aH^{2}$ is contractively contained in $\HH(b)$, i.e., for every $g\in H^{2}$, 
\begin{equation}
\label{M(a)}
\Vert ag\Vert_{b}\leq \Vert ag\Vert_{\mathscr{M}(a)} = \Vert g\Vert_{2}.
\end{equation}
In particular, $a\in\HH(b)$ and thus, if $\mu$ is a Carleson measure for $\HH(b)$, then necessarily $a$ is $\mu$-admissible.  

\begin{Proposition} \label{P-direct-a}
Let $b\in H^\infty_1$ be  non-extreme and let $\mu \in M_{+}(\D^{-})$ be a Carleson measure for $\mathscr{H}(b)$. Then $d \nu = |a|^2 d \mu$ is a Carleson measure for $H^{2}$.
\end{Proposition}
\begin{proof} 
For $\lambda \in \D$ and $k_{\lambda}$ the standard reproducing kernel for $H^2$, apply the direct embedding inequality to  $ak_{\lambda}$ and use \eqref{M(a)} to obtain
\[
\Vert k_{\lambda}\Vert_{2}=\Vert ak_{\lambda}\Vert_{\mathscr{M}(a)}
 \ge \|ak_{\lambda}\|_b \gtrsim \Vert ak_{\lambda}\Vert_{\mu}
 =\Vert k_{\lambda}\Vert_{\nu},\quad \lambda\in\D,
\]
which implies, by the reproducing kernel thesis for $H^2$, that $\nu$ is a Carleson measure for $H^ 2$.
\end{proof}

\begin{Remark}
We will see in a moment that without additional assumptions on $a$ and $b$, the converse is not always true. 
\end{Remark}

For our next set of results we need some additional facts concerning $\HH(b)$ spaces.  For $\alpha \in \T$ let $\sigma_\alpha$ be the  Aleksandrov--Clark measure \cite{CRM, Clark72, PolSa} associated with the function $\bar\alpha b$, i.e., $\sigma_{\alpha}$ is (via a classical theorem of Herglotz) the unique positive measure on $\T$ satisfying
\begin{equation} \label{AC-sigma-alpha}
\frac{1-|b(z)|^2}{|1-\bar\alpha b(z)|^2}=\int_\T \frac{1-|z|^2}{|1-\bar \zeta z|^2}\,d\sigma_\alpha(\zeta),\quad z\in\D.
\end{equation}
According to \cite[IV-10]{Sa}, since $b$ is non-extreme,  $\sigma_\alpha \ll m$ for $m$-almost every $\alpha\in\T$. 

Now let 
$$F_\alpha:=\frac{a}{1-\bar\alpha b}$$
and note that $F_{\alpha}$ belongs to $H^2$ and is outer.  From \cite[IX-4]{Sa} we know that
$\mathscr{M}(a)=\HH(b)$ (with equivalent norms) if and only if there is an $\alpha\in\T$ such that  $\sigma_\alpha \ll m$ and $|F_\alpha|^2\in (A_2)$.  

The following theorem may seem overly technical at first glance but it will have a very useful corollary (See Corollary \ref{Cor:rationnel} below).

\begin{Theorem}\label{Thm:direct-embedding}
Let $(a,b)$ be a Pythagorean pair, $\mu \in M_{+}(\D^{-})$, $\alpha\in\T$  such that $\sigma_\alpha \ll m$. Assume  there exists a polynomial $p$ having all of its roots in $\T$ and an $f \in H^2$ satisfying the conditions $|f|^2\in (A_2)$ and $F_\alpha=pf$. Then the following 	assertions are equivalent: 
\begin{enumerate}
\item The measure $\mu$ is a Carleson measure for $\HH(b)$;
\item The function $a$ is $\mu$-admissible and the measure $|a|^2\,d\mu$ is a Carleson measure for $H^2$. 
\end{enumerate}
\end{Theorem}

\begin{proof}
The implication $(1)\Rightarrow (2)$ has already been proved. Let us now focus on the reverse implication and assume that $d\nu:=|a|^2 d\mu$ is a Carleson measure for $H^2$. Write 
\[
p(z)=\prod_{i=1}^s(z-\zeta_i)^{m_i},
\]
where, by hypothesis, $\zeta_i\in\T$. Let 
$$N=m_1+m_2+\dots+m_s$$ denote the degree of $p$. According to \cite[X-18]{Sa}, we know that $\mathscr{M}(a)$ is closed in $\HH(b)$ with co-dimension $N$. 

If $N=0$ then $|F_\alpha|^2\in (A_2)$ and we know that $\mathscr{M}(a)=\HH(b)$ with equivalent norms. Then for every $f=ag\in\HH(b)$, we have 
\[
\|f\|_\mu=\|g\|_\nu\lesssim \|g\|_2=\|ag\|_{\mathscr{M}(a)}\asymp\|f\|_b,
\]
which proves the desired embedding. 

Now assume that $N\geq 1$ and let us first show that  $\HH(b)$ can be written as
\begin{equation} \label{star-decomp}
\HH(b)=\mathscr{M}(a) \dotplus \mathcal P_{N-1},
\end{equation}
%where $\mathcal P_{N-1}$ denotes the set of polynomials of degree less or 
%equal to $N-1$ and 
where the sum %$\dotplus$ 
in the above decomposition is direct (not necessarily orthogonal). First note that since $b$ is  non-extreme, the polynomials belong to $\HH(b)$. Now let $q\in \mathcal P_{N-1}\cap \mathscr{M}(a)$. That means that the polynomial $q$ can be written as $q=ag$ for some $g\in H^2$. But then, since
$$pf= \frac{a}{1-\ov{\alpha}b},$$ we see that the rational function 
$$\frac{q}{p}=(1-\bar\alpha b)fg$$ belongs to $H^1$. This is clearly
possible if and only if the poles of $q/p$ are outside $\D^{-}$. In particular, we see that the polynomial $q$ should have a zero of order at least $m_i$ at
each point $\zeta_i$. Since the degree of $q$ is less or equal to $N-1$, this necessary implies that $q=0$. Hence the sum $\mathscr{M}(a)\dotplus \mathcal P_{N-1}$ is direct. Now since $$\mbox{dim }\mathcal P_{N-1}=N=\mbox{codim }\mathscr{M}(a),$$ we obtain \eqref{star-decomp}.
 
In particular, the angle between the subspaces $\mathscr{M}(a)$ and $\mathcal P_{N-1}$ is strictly positive which means that
\[
\|f\|_b\asymp \|ag\|_{b}+\|p\|_b,
\]
for every $f=ag+p\in \HH(b)$ where $ag\in\mathscr{M}(a)$ and $p\in\mathcal P_{N-1}$. Moreover, 
%note that according to the theorem of isomorphism of Banach, 
since $\mathscr{M}(a)$ is a closed subspace of $\HH(b)$, contractively embedded
($\|ag\|_b\leq \|ag\|_{\mathscr{M}(a)}$, $g\in H^2$), the open mapping theorem
shows %yields a constant $c>0$ such 
that 
\[
\|ag\|_{\mathscr{M}(a)}\lesssim  \|ag\|_b,\quad g\in H^2.
\]

Since $\nu$ is a Carleson measure for $H^2$, we have, for $ag\in~\mathscr{M}(a)$, 
\begin{eqnarray*}
\|ag\|_\mu&=&\|g\|_\nu\lesssim \|g\|_2=\|ag\|_{\mathscr{M}(a)}\\
&\lesssim & \|ag\|_b.
\end{eqnarray*}
On the other hand, if $p(z)=\sum_{n=0}^{N-1}a_nz^n \in \mathcal P_{N-1}$, we see, since $\mu$ is a finite measure, that
\begin{eqnarray*}
\|p\|_\mu &\leq &\sum_{n=0}^{N-1} |a_n|\|z^n\|_\mu\lesssim \sum_{n=0}^{N-1}|a_n| \\
&\lesssim & \left(\sum_{n=0}^{N-1}|a_n|^2\right)^{1/2}\\
&=& \|p\|_2\leq \|p\|_b.
\end{eqnarray*}
We conclude that for $f=ag+p$, where $ag\in\mathscr{M}(a)$ and $p\in\mathcal P_{N-1}$, we have 
\[
\|f\|_\mu=\|ag+p\|_\mu\leq \|ag\|_\mu+\|p\|_\mu \lesssim \|ag\|_b+\|p\|_b
\asymp \|f\|_b. \qedhere
\]  
\end{proof}

A nice application of Theorem \ref{Thm:direct-embedding} is the following corollary. Note that if $b$ is a rational function then so is $a$ \cite[Remark 3.2]{MR2418122} so there is no need to impose any $\mu$ admissibility conditions. 

\begin{Corollary}\label{Cor:rationnel}
Let $b$ be a rational and non-extreme  and $\mu \in M_{+}(\D^{-})$.  Then the following assertions are equivalent: 
\begin{enumerate}
\item The measure $\mu$ is a Carleson measure for $\HH(b)$;
\item The measure $|a|^2\,d\mu$ is a Carleson measure for $H^2$. 
\end{enumerate}
\end{Corollary}

\begin{proof}
According to Theorem~\ref{Thm:direct-embedding}, it is sufficient to prove that there exists an $\alpha\in\T$ such that $\sigma_{\alpha}$ from \eqref{AC-sigma-alpha} satisfies $\sigma_{\alpha} \ll m$, a polynomial $p$  having all of its roots on $\T$, and a function $f \in H^2$ with $|f|^2\in (A_2)$ such that $F_\alpha=pf$.

We first observe that the associated outer function $a$ is also a rational function (see \cite[Remark 3.2]{MR2418122}).
% First note that the associated outer function $a$ is also a rational function. Indeed,  write $b(z)=p(z)/r(z)$ where $p$ and $r$ are two polynomials with $PGCD(p,r)=1$. Then necessarily $r(z)\neq 0$ for every $z\in\D^{-}$ and in particular, we may also assume that $r(0)>0$.  By the theorem of Fej\'er--Riesz, there exists a polynomial $q$ without zeros in $\D$ such that $|r|^2-|p|^2=|q|^2$ on $\T$. We may also assume that $q(0)>0$. Then we set $a=q/r$. It is clear that $a$ is an outer function which satisfies $|a|^2+|b|^2=1$ on $\T$. By uniqueness, the function $a$ is the desired function such that $(a,b)$ forms a  Pythagorean pair (actually a corona pair since $a,b$ are continuous on $\D^-$,$a$ is outer rational and $b$ can only vanish in a finite number of points in $\D$).
Write $a=q/r$ where $q$ and $r$ are two polynomials with $GCD(q,r)=1$. 
Then, necessarily, $r(z)\neq 0$ for every $z\in\D^{-}$ and let us denote by $\zeta_i$, $1\leq i\leq N$, the zeros of $q$ on $\T$.  Note that 
since these zeros are the same as those of $a$ we actually have
\[
\{z\in\T:|b(z)|=1\}=\{\zeta_1,\dots,\zeta_N\}.
\]
Now choose $\alpha\in\T\setminus\{b(\zeta_1),\dots,b(\zeta_N)\}$ such that $\sigma_\alpha \ll m$ (which is always possible because $\sigma_\alpha \ll m$ for $m$-almost every $\alpha\in\T$). Moreover, according to the choice of $\alpha$, the function $1-\bar\alpha b$, which is continuous on $\D^{-}$, cannot vanish and hence
\begin{equation}\label{eq:r-1-baralpha-b}
\inf_{z\in \D^{-}}|r(z)(1-\bar\alpha b(z))|>0. 
\end{equation}
It remains to factor the polynomial $q$ as $q_1q_2$ where $q_1$ has all of its roots on $\T$ and $q_2$ has all of its roots outside $\D^-$. Then 
\[
F_\alpha=\frac{a}{1-\bar\alpha b}=q_1 \frac{q_2}{r(1-\bar\alpha b)}=q_1 f,  
\]
where $$f=\frac{q_2}{r(1-\bar\alpha b)}.$$ In view of \eqref{eq:r-1-baralpha-b}, we easily see that $f$ and $1/f$ are continuous on $\D^{-}$ which implies $|f|^2\in (A_2)$. 

The proof is completed by an application of Theorem~\ref{Thm:direct-embedding}. 
\end{proof}

\begin{Remark} \label{Ransford}
As a byproduct of our proof of Corollary~\ref{Cor:rationnel} and Theorem~\ref{Thm:direct-embedding}, we see that if $b$ is rational and  non-extreme  and if $\zeta_1,\dots,\zeta_n$ are the zeros of $a$ on $\T$, listed according to multiplicity, then
\[
\HH(b)=\left(\prod_{j=1}^n (z-\zeta_i) \right)H^2\dotplus \mathcal P_{n-1}.
\]
This decomposition already appears in \cite[Lemma 4.3]{ransford}  but their argument is based on a difficult result of  Ball and Kriete which gives a condition as to when one $\HH(b)$-space is contained in another. 

Moreover, if we gather \cite[Theorem 4.1.]{ransford} and Corollary~\ref{Cor:rationnel}, then we recover a result of \cite{fricain-shabankhah} concerning the characterization of direct Carleson measures for a Dirichlet-type space associated with a finite sum of Dirac measures. 
\end{Remark}

We already mentioned in Section \ref{section2} that if $(a,b)$ is a corona pair and if $T_{\ov{a}/a}$ is invertible then $\HH(b)=\mathscr{M}(a)$. We also pointed out that this is equivalent to the fact that there exists $\alpha\in\T$ with $\sigma_\alpha$ (from \eqref{AC-sigma-alpha}) satisfies $\sigma_{\alpha} \ll m$ and $|F_\alpha|^2\in (A_2)$. In particular, 
applying Theorem~\ref{Thm:direct-embedding} to this situation (or by direct 
inspection) immediately gives the following result.

\begin{Corollary} \label{C-CP}
Suppose that $(a,b)$ forms a corona pair and $T_{\ov{a}/a}$ is invertible. If $\mu \in M_{+}(\D^{-})$, then the following assertions are equivalent.
\begin{enumerate}
\item The measure $\mu$ is a Carleson measure for $\HH(b)$;
\item The function $a$ is $\mu$-admissible and the measure $|a|^2\,d\mu$ is a Carleson measure for $H^2$. 
\end{enumerate}
% the embedding (\ref{direct embed}) holds if and only if $\nu$  is a Carleson measure for %$H^{2}$.
\end{Corollary}
%
%\begin{proof}
%Note first that the necessary condition has already been proved in a more general case with the previous proposition. We have the following equivalences:
%\begin{eqnarray*}
%\MM(a)\hookrightarrow L^{2}(\mu) & \Longleftrightarrow & \Vert f\Vert_{\mu}\lesssim\Vert f\Vert_{\MM(a)},\quad  f=ag\in \MM(a)=aH^2,\\
%& \Longleftrightarrow & \Vert ag\Vert_{\mu}\lesssim\Vert g\Vert_{2},\quad  g\in H^{2},\\
%& \Longleftrightarrow & \nu \text{ is a Carleson measure for $H^2$}.
%\end{eqnarray*}
%Since $\MM(a)=\HH(b)$, this completes the proof.
%\end{proof}
\begin{Remark}
 A sufficient condition for direct Carleson measures for $\HH(b)$ is given in \cite{BFM}. More precisely, for $\varepsilon\in (0,1)$, we let $\Omega(b,\varepsilon)$ denote the sub-level set 
\[
\Omega(b,\varepsilon):=\{z\in\D:|b(z)|<\varepsilon\},
\]
and 
\[
\sigma(b):=\{\zeta\in\T:\liminf_{z\to\zeta}|b(z)|<1\},
\]
denote the boundary spectrum of $b$
and we set 
$$\widetilde\Omega(b,\varepsilon)=\Omega(b,\varepsilon)\cup\sigma(b).$$
It is known \cite[Theorem 6.1]{BFM} that if $\mu$ satisfies 
\[
\mu(S(I))\lesssim |I|,
\]
for any arc $I$ such that $S(I)\cap\widetilde\Omega(b,\varepsilon)\neq\emptyset$, then 
$$\|f\|_{\mu}\lesssim \|f\|_b, \quad f\in \HH(b).$$ 

Note that our results, in some sense, complete the picture because one can, using Corollary~\ref{Cor:rationnel}, produce an example of a Carleson measure for $\HH(b)$ where the above criterion cannot be applied.
Indeed let $$b(z)=\frac{1+z}{2}.$$
Then $$a(z)=\frac{1-z}{2}$$ and $\sigma(b)=\T\setminus\{1\}$. Thus checking the Carleson condition for $\mu$ on Carleson squares \eqref{window} which intersect  $\widetilde \Omega(b,\varepsilon)$ is, in this case, equivalent to saying that $\mu$ is a Carleson measure for $H^2$. But now it is easy to construct an example of a measure $\mu$ which is Carleson for $\HH(b)$ but not for $H^2$. For instance, we can consider the measure $\mu$ on the interval $(0,1)$ defined by 
$$d\mu(t)=(1-t)^{-\beta}dt,$$for $\beta\in (0,2]$. According to Corollary~\ref{Cor:rationnel}, $\mu$ is a Carleson measure for $\HH(b)$ but cannot be a Carleson measure for $H^2$ because if we consider the arc $I_{\vartheta}:=(e^{-i\vartheta},e^{i\vartheta})$, $\vartheta \in (0, \pi/2)$, %and $h=2\vartheta$, 
we have 
\[
\mu(S(I_{\vartheta}))=\int_{1-\vartheta/2\pi}^1\frac{dt}{(1-t)^{\beta}}=\frac{(\vartheta/2\pi)^{1-\beta}}{1-\beta},
\]
and thus 
\[
\sup_{\vartheta>0}\frac{\mu(S(I_{\vartheta}))}{|I_{\vartheta}|} = \infty.
\]

However, it is important to note here that
our results focus on the case where $b$ is  non-extreme  whereas in \cite{BFM}, there are no such assumptions on $b$. 
\end{Remark}

\begin{Remark}
In \cite{BFM}, it is shown that for a particular class of functions $b$, the reproducing kernel thesis is true. More precisely, let $b\in H^{\infty}_{1}$ and assume that there exists $\varepsilon\in (0,1)$ such that $\Omega(b,\varepsilon)$ is connected and its closure contains the spectrum $\sigma(b)$. If $\|k_\lambda^b\|_\mu\lesssim \|k_\lambda^b\|_b$ holds for every $\lambda\in\mathbb D$, then $\|f\|_\mu\lesssim\|f\|_b$ for every $f\in\mathscr H(b)$. However, in \cite{MR1945290}, F. Nazarov and A. Vol'berg showed that this is no longer true in the general case (their counterexample corresponds to the case where $b$ is an inner function). According to the results of this paper, it would be natural to conjecture that the reproducing kernel thesis is true in the non-extreme case. However, we currently do not know how to prove this.
\end{Remark}

\section{Examples} \label{section6}
We would like to  
%now investigate 
discuss the necessity of the two hypotheses appearing in Corollary \ref{C-CP}. To do so, we will construct two examples where either of the two
conditions
\begin{equation}\label{condition i}
(a,b) \text{ is a corona pair}
\end{equation}
\begin{equation}\label{condition ii}
T_{\ov{a}/a} \text{ is invertible}
\end{equation}
are violated and yet $\nu$ is a Carleson measure for $H^2$ for which there is no Carleson embedding $\HH(b)\hookrightarrow L^2(\mu)$. Let us start with condition (\ref{condition i}).

\begin{Example}
There is a Pythagorean pair $(a,b)$ and a $\mu\in M_+(\D^-)$ 
such that  $(a,b)$ is not a corona pair, $T_{\ov{a}/a}$ is
invertible, $\nu$ is a Carleson measure for $H^2$, yet $\mu$ is not a Carleson
measure for $\HH(b)$.

To see this, let 
$$a(z)=c(1-z)^{\alpha}, \quad \alpha\in (0,1/2),$$ guaranteeing that $|a|^2 \in (A_2)$ (equivalently $T_{\ov{a}/a}$ is invertible).
Here $c=2^{-\alpha}$ so that $\|a\|_{\infty}=1$. Clearly
$1-|a|^2$ is a bounded, $\log$-integrable function, and so that there is an 
outer function $b_0\in H^{\infty}$ 
such that $|a|^2+|b_0|^2=1$ a.e.\ on $\T$. Symmetrically, $b_0$ is
also non-extreme, and its Pythagorean mate is $a$.
%Then, $|b_0(z)|^2=1-c^2|1-z|^{\alpha}$, $z\in \T$, is in some H\"older class
%so that by standard results on the Hilbert transform, the argument of $b_0$ 
%will also be in that class. We deduce that $|1-b_0(z)|\simeq |1-z|^{\alpha}$.
Now consider the Blaschke product $B=B_{\Lambda}$ whose zeros are 
%interpolating 
$\Lambda=\{1-1/2^n\}_{n \geq 1}$, and set $b=Bb_0$. Then $(a,b)$ is
not a corona pair since 
$$|a(\lambda_n)|+|b(\lambda_n)|=
|a(\lambda_n)|\to 0, \quad n\to\infty.$$

%From de Branges-Rovnyak
%space theory we know that
%\[
% \HH(B)=\HH(B)\oplus B\HH(b_0)=K_B\oplus \HH(b_0),
%\]
%where $K_B=(BH^2)^{\perp}$ is the model space.

Now consider the measure $\mu$ on $\T$ defined by
\[
 d\mu(z)=\frac{1}{|1-z|^{\beta}}dm(z)
\] 
for some $0<\beta\le 2\alpha$. Then
\[
 d\nu=|a|^2d\mu=c^2|1-z|^{2\alpha-\beta}dm
\]
has bounded Radon-Nikodym derivative and is thus a
Carleson measure for $H^2$.

To show that $\mu$ is not a Carleson measure for $\mathscr{H}(b)$, we now estimate the $L^2(\mu)$-norms of
the normalized reproducing kernels 
\[
 \kappa_{\lambda_n}(z)=\frac{k_{\lambda_n}(z)}{\|k_{\lambda_n}\|_2}
 =\frac{\sqrt{1-|\lambda_n|^2}}{1-\ov{\lambda}_nz}, \quad z\in\D.
\]
Indeed,
\begin{eqnarray*}
 \|\kappa_{\lambda_n}\|_{L^2(\mu)}^2
 &=&\int_{\T}\frac{1-|\lambda_n|^2}{|\zeta-\lambda_n|^2}
 \frac{dm(\zeta)}{|1-\zeta|^{\beta}}
 \simeq \frac{1}{2^n}\int_{-\pi}^{\pi}
 \frac{1}{|e^{it}-\lambda_n|^2}\frac{1}{t^{\beta}}dt\\
 &\ge&\frac{1}{2^n}\int_{-1/2^n}^{1/2^n}\frac{1}{|e^{it}-\lambda_n|^2}
 \frac{1}{t^{\beta}}dt
  \simeq \frac{1}{2^n} \times 2^{2n}\int_{-1/2^n}^{1/2^n}\frac{1}{t^{\beta}}dt\\
 &\simeq& \frac{1}{2^n}\times 2^{2n} \times \frac{1}{2^{n(1-\beta)}}
 =2^{n(2-1-(1-\beta))}\\
 &=& 2^{n\beta} \to \infty,\quad n\to \infty.
\end{eqnarray*}
By \eqref{eq:norm-Hb-kernel}, since $b(\lambda_n)=0$, $n\geq 1$, we have 
\[
\|\kappa_{\lambda_n}\|_b=1.
\]
As a result, $\mu$ is not a Carleson measure for $\HH(b)$.
\end{Example}

The following result discusses the necessity of condition (\ref{condition ii}).

\begin{Theorem}\label{thm3.4}
Let $(a,b)$ be a corona pair such that $|a|^{-2}\in L^1$ and let 
$d\mu=|a|^{-2}dm$. Then $\mu$ is a Carleson measure for $\HH(b)$ if and only if $|a|^2 \in (A_{2})$.
\end{Theorem}

\begin{proof}
The following equivalences are quite obvious:
\begin{eqnarray*}
\mathscr{M}(\ov{a})\hookrightarrow L^{2}(\mu) & \Leftrightarrow & \Vert f\Vert_{\mu}\lesssim\Vert f\Vert_{\mathscr{M}(\ov{a})},\quad  f\in \mathscr{M}(\ov{a}),\\
& \Leftrightarrow & \Vert T_{\ov{a}}g\Vert_{\mu}\lesssim\Vert g\Vert_{2},\quad  g\in H^{2},\\
& \Leftrightarrow & T_{\ov{a}} : H^{2}\to L^{2}(\mu) \text{ is bounded}\\
& \Leftrightarrow & T_{\ov{a}} : L^{2}\to L^{2}(\mu) \text{ is bounded}
\end{eqnarray*}
Let $j:L^2\longrightarrow L^2(\mu)$ be the onto isometry
$j(f)=\overline{a}f$ and observe that the following diagram commutes:
\begin{diagram}
L^2 &  &  \\ %\rTo & L^2 \\
\dTo_{j} & \rdTo^{T_{\ov{a}}} & \\ % \dTo_{j_2}\\
L^2(\mu) & \rTo^{      \phantom{asdasd}P_+\phantom{asdasd}} & L^2(\mu).
\end{diagram}
Hence
\begin{eqnarray*}
\mathscr{M}(\ov{a})\hookrightarrow L^{2}(\mu) 
& \Leftrightarrow & P_+ : L^{2}(\mu)\to L^{2}(\mu) \text{ is bounded}\\
& \Leftrightarrow & |a|^2 \in (A_{2}).
\end{eqnarray*}
Remembering that since $(a,b)$ is a corona pair, then $\HH(b)=\mathscr{M}(\ov{a})$ with equivalent norms which completes the proof.
\end{proof}

\begin{Example}
Based on Theorem \ref{thm3.4},
we now construct an explicit example of a corona pair $(a,b)$ so that $d\mu=|a|^{-2}dm$ is not a Carleson measure for $\HH(b)$ whereas $|a|^{2}d\mu=dm$ is naturally a Carleson measure for $H^2$. In view of Corollary \ref{C-CP}, condition (\ref{condition ii}) will not be satisfied here.
%It is thus sufficient to pick $a$ such that $|a|^2$ does not satisfy the
%Muckenhoupt condition. 

Let $\lbrace\beta_n\rbrace_{n\geq1}$ %and $(\beta_n)_n$ 
be a sequence decreasing
to zero and bounded by $1/2$. 
%Let $a$ be the outer function satisfying
For $n\ge 1$, introduce the intervals
\[
 I_n:=\Big[\frac{1}{2^{2n+1}}+\frac{1}{2^{3n}},\frac{1}{2^{2n}}
 -\frac{1}{2^{3n}}\Big),\quad
 J_n:=\Big[\frac{1}{2^{2n}}+\frac{1}{2^{3n}},\frac{1}{2^{2n-1}}
 -\frac{1}{2^{3n}}\Big)
\]
and 
define a symmetric function $u$  ($u(e^{it})=u(e^{-it})$) on these intervals by
%for every $n\ge 0$
\[
 u(e^{it})=
 \left\{
 \begin{array}{ll}
 1/2 & \text{if }t\in J_n,\\
 \beta_n &\text{if }t\in I_n.
 \end{array}
 \right. 
\]
Connect smoothly and monotonically the values $1/2$ and $\beta_n$ between
these intervals
and  set $u(e^{it})=1/2$ on the remaining part of the circle. We need to impose
two conditions on $\lbrace\beta_n\rbrace_{n\geq1}$. First suppose that
\[
 \sum_{n\ge 1}\frac{1}{2^{2n}} %(\log\alpha_n^{-1}+
 \log\beta_n^{-1}<\infty
\]
guaranteeing that $u$ is log-integrable. %Let $u_1$ the function
Hence there is an outer function $a$ given by
\[
 a(z)=\exp\left(\frac{1}{2}\int_{-\pi}^{\pi}\frac{e^{it}+z}{e^{it}-z}
 \log u(e^{it})\frac{dt}{2\pi}\right),\quad z\in\D.
\]
Obviously $|a|^2=u$ a.e.\ on $\T$.
Secondly we impose the condition
\[
  \sum_{n\ge 0}\frac{1}{2^{2n}} %(\log\alpha_n^{-1}+
 \frac{1}{\beta_n}<\infty
\]
which guarantees that $|a|^{-2}=u^{-1}\in L^1(\T)$. One can pick,  for instance,
$\beta_n\ge 2^{-2\alpha n}$ with $\alpha<1$. 

Now, in view of the smoothness and non-vanishing properties of $u$,
the function $a$ extends continuously from $\D$ to every
point $\zeta\in\T\setminus \{1\}$.
% such that $|A|=u$ a.e.\ on $\T$.
%Set $a=\sqrt{A}$.
%The only point where $u$ is discontinuous with vanishing limit inf is 
%in 1 so that this is the only
%point where $a$ is non-smooth.
%By the analogous definition of $b$ we have
Note that $a$ satisfies \eqref{extreme-defn}. Thus we can consider an outer function $b$ such that $|b|^2=1-|a|^2$ almost everywhere on $\T$. Then 
we have 
\[
 \log|b(z)|=\frac{1}{2}\int_{-\pi}^{\pi} \frac{1-|z|^2}{|e^{it}-z|^2}
 \log({1-u(e^{i t})})\frac{dt}{2\pi},\quad z\in\D.
\] 
Now if $(a,b)$ were not a corona pair, then there would be a sequence 
$\lbrace z_n\rbrace_{n\ge 1}$ such that $a(z_n)\to 0$ and $b(z_n)\to 0$. Clearly we can
suppose $z_n\to \zeta\in\T$. The only point where this can happen
for $a$ is $\zeta=1$. Now, fix an arc of $\T$ which contains $1$.
There, the function $b$
is bounded below by $1/2$ so that the outer function $b$ cannot be small
in this neighborhood (the harmonic function $\log|b|$ cannot tend to
$-\infty$ there). Hence $b(z_n)$ can not approach $0$. We conclude that $(a,b)$ is a corona pair.

Let us check that $|a|^2$ is not $(A_2)$.
Consider the intervals
\[
 K_n:=\Big[\frac{1}{2^{2n+1}}+\frac{1}{2^{3n}},\frac{1}{2^{2n-1}}
 -\frac{1}{2^{3n}}\Big)\supset I_n\cup J_n,
\] 
the length of which is comparable to $2^{-2n}$.
%On $K_n$ the function $a$ 
We have
\[
 \frac{1}{m(K_n)}\int_{K_n}|a(e^{it})|^2 dt
 \simeq 2^{2n}\int_{J_n} u(e^{it})dt\simeq 1
\]
(the contribution of $u$ on $I_n$ is negligible and $m(J_n)\simeq 2^{-2n}$).
On the other hand,
\[
 \frac{1}{m(K_n)}\int_{K_n}\frac{1}{|a(e^{it})|^2} dt
 \simeq 2^{2n}\int_{I_n} \frac{1}{u(e^{it})}dt\simeq \frac{1}{\beta_n}
\]
(the contribution of $1/u$ on $J_n$ is negligible and $m(I_n)\simeq
2^{-2n}$).
Hence, as $n \to \infty$,
\[
\left(\frac{1}{m(K_n)}\int_{K_n}|a(e^{it})|^2 dt\right)
 \left( \frac{1}{m(K_n)}\int_{K_n}\frac{1}{|a(e^{it})|^2} dt\right)
 \simeq \frac{1}{\beta_n}
 \rightarrow \infty,
\]
which proves the claim. Then according to Theorem~\ref{thm3.4} $d\mu=|a|^{-2}\,dm$ is not a Carleson measure for $\HH(b)$. 
%Let us be more concrete:
%\begin{eqnarray*}
%%\alpha_n=\frac{1}{2^{\alpha 2n}},\quad
%\beta_n=\frac{1}{2^{\beta 2n}}
%\end{eqnarray*}
%where $1<\alpha<\beta$. In particular, close to $1=e^{i0}$, 
%we have $u(e^{it})\le t^{\alpha}$ making the function differentiable
%in $1$. Then there exists an outer function $A$
%
% let for instance $a(z)=c(1-z)^{\alpha}$, $1/2<\alpha<1$,
%which is moreover non-extreme. We obtain an outer function $b_0$, $|a|^2+|b_0|^2=1$
%a.e.\ $\T$, which by symmetry is non-extreme. The pair $(a,b_0)$ is a corona 
%pair (indeed, $a$ is small only close to 1, where the outer function $b_0$ has
%big modulus on the boundary, so that its harmonic extension cannot be arbitrarily
%small there). Now, by the above theorem, we do not have a Carleson
%embedding $\HH(b)\subset L^2(\mu)$, still $d\nu=|a|^2d\mu=dm$ is
%obvioulsy Carleson for $H^2$.
\end{Example}

%From known results (ref ??) about Riesz projection between two different weighted $L^{2}$ 
%spaces, we can state some kind of slightly improved result.
To finish this section we comment further on the general situation $d\mu=hdm$
when $h$ is not necessarily equal to $|a|^{-2}$.
Then the direct embedding result is connected with the so-called two-weighted
estimates. Indeed, with a similar argument as in Theorem \ref{thm3.4}, we
can show that when $(a,b)$ is a corona pair such that $|a|^{-2}\in L^1$ and $d\nu=|a|^2d\mu$ for some absolutely continuous
measure $d\mu=h dm$ on $\T$, then
\begin{eqnarray*}
\mathscr{M}(a) \hookrightarrow L^{2}(\mu) & \Leftrightarrow & \|f\|_{\mu}\lesssim \|f\|_{\mathscr{M}(\ov{a})},\quad  f\in \mathscr{M}(\ov{a}),\\
& \Leftrightarrow & \Vert T_{\ov{a}}g\Vert_{\mu}\lesssim\Vert g\Vert_{2},\quad  g\in H^{2},\\
& \Leftrightarrow & T_{\ov{a}} : H^{2}\to L^{2}(\mu) \text{ is bounded}\\
& \Leftrightarrow & T_{\ov{a}} : L^{2}\to L^{2}(\mu) \text{ is bounded}.
%& \Longleftrightarrow & P_{+} : L^{2}(\mu)\to L^{2}(\mu) \text{ is bounded}\\
%& \Longleftrightarrow & |a| \text{ satisfies } (A_{2}).
\end{eqnarray*}
As was done earlier,  let us consider the corresponding commutative diagram.
To do this set $d\gamma= |a|^{-2}dm$ and  $j:L^2\to L^2(\gamma)$ be
the onto-isometry $j(f):=\ov{a}f$. Then the following diagram commutes:
\begin{diagram}
L^2 &  &  \\ %\rTo & L^2 \\
\dTo_{j} & \rdTo^{T_{\ov{a}}} & \\ % \dTo_{j_2}\\
L^2(\gamma) & \rTo^{      \phantom{asdas}P_+\phantom{asdas}} & L^2(\mu)
\end{diagram}
Hence, recalling that $d\mu=h\,dm$,
%begin{eqnarray*}
 $\mathscr{M}(\ov{a})$ embeds into $L^{2}(\mu)$ if and only if 
$$ P_+ : L^{2}\Big(|a|^{-2}dm\Big)\to L^{2}(h\,dm) \text{ is bounded}.$$
Using one more time that $\mathscr{M}(\ov{a})=\HH(b)$ (with equivalent norms), we get:
\begin{Theorem}\label{UselessThm}
Let us suppose that $(a,b)$ forms a corona pair such that $|a|^{-2}\in L^1$. Let $d\mu=h\,dm$ with $h\in L^1$. 
Then $\mu$ is a Carleson measure for $\HH(b)$ if and only if 
%$|a|^2$ satisfies the Muckenhoupt condition $(A_{2})$.
$P_+ : L^{2}\Big(|a|^{-2}dm\Big)\to L^{2}(h\,dm)$ is bounded
\end{Theorem}

It is known that the generalized Muckenhoupt condition is necessary
for the continuity of $P_+$ (or the Hilbert transform) between two
weighted $L^2$-spaces:
\begin{equation}
\label{GenMuckenhoupt}
\sup_{I} \left(\frac{1}{m(I)}\int_{I}h\,dm\right)\left(\frac{1}{m(I)}\int_{I}|a|^2\,dm\right)<\infty,
\end{equation}
but this condition is, in general, not sufficient (see for instance \cite[p.154]{Duo}).

\section{Norm equivalence and isometric embeddings}\label{section7}

This next to last section is devoted to a discussion of equivalent norms on
$\HH(b)$ and isometric embeddings of $\HH(b)$ in $L^2(\mu)$. It will turn out that the $L^2(\mu)$ norm is equivalent to the $\HH(b)$-norm only when
 $\HH(b)= \mathscr{M}(a)$.
 
 Let us make the following observation. If we are to impose the condition that the $L^2(\mu)$ norm is equivalent to the $\mathscr{H}(b)$ norm then we require that $\mathscr{H}(b)_{\mu} = \mathscr{H}(b)$ and that $\|f\|_{\mu} \asymp \|f\|_{b}$ for all $f \in \mathscr{H}(b)$. When $b$ is  non-extreme  then Proposition \ref{bainh2} and Theorem \ref{Sarasonequivalence} imply that $H^{\infty} \subset \mathscr{H}(b)$. This means that if every function in $\mathscr{H}(b)$ has a non-tangential limit $\mu$-almost everywhere on $\T$ then every function in $H^{\infty}$ must also have this property. By a classical result of  Lusin \cite[p.~24]{CL} (given any closed subset of $\T$ with Lebesgue measure zero there is an $f \in H^{\infty}$ which does not have non-tangential limits on this set), we see that the singular part of $\mu|\T$ (with respect to $m$) is zero. 

\begin{Proposition}\label{Prop:equivalent-norms}
Let $b\in H^\infty_1$ be  non-extreme point, $\mu\in M_+(\D^{-})$, and $h=d\mu|_{\T}/dm$. 
Assume that $\HH(b)_\mu=\HH(b)$ and 
$$\|f\|_\mu\asymp \|f\|_b, \quad f\in\HH(b).$$
 Then $(a,b)$ is a corona pair and $\vert a\vert^2 \in (A_2)$. In other words, it is necessary that $\HH(b)=\mathscr{M}(a)$.
\end{Proposition}

\begin{proof}
First note that by our above discussion we have  $d \mu|\T = h dm$. Second, according to Theorem~\ref{MainThm}, we have 
\begin{equation}\label{eq:reverse-condition-a}
\delta:=\mathrm{ess} \inf_{\T}|a|^2 h>0.
\end{equation}
Third,  using \eqref{eq:norm-Hb-kernel} there exists a constant $c>0$ such that 
\[
c \int_{\D^{-}}\frac{1}{|1-\bar\lambda z|^2}\,d\mu(z)\leq \frac{|a(\lambda)|^2+|b(\lambda)|^2}{|a(\lambda)|^2(1-|\lambda|^2)}
\]
for every $\lambda\in\D$. Hence with \eqref{eq:reverse-condition-a} we get 
\[
c\delta\int_\T |a(\zeta)|^{-2} \frac{1-|\lambda|^2}{|1-\bar\lambda \zeta|^2}\,dm(\zeta)\leq \frac{|a(\lambda)|^2+|b(\lambda)|^2}{|a(\lambda)|^2}.
\]
Now subharmonicity of $|a|^{-2}$ gives 
\[
\frac{c\delta}{|a(\lambda)|^2}\leq \frac{|a(\lambda)|^2+|b(\lambda)|^2}{|a(\lambda)|^2}\leq\frac{(|a(\lambda)|+|b(\lambda)|)^2}{|a(\lambda)|^2} ,
\]
which proves that $(a,b)$ is a corona pair. 
Let us now  prove that $\vert a\vert^2 \in (A_2)$. From Theorem \ref{UselessThm}
%the discussion at the end of Section \ref{section5}, 
we know that 
$$P_+:L^{2}(|a|^{-2}dm)\to L^{2}(h dm)$$ is 
bounded (note that $0\le |a|^{-2}\leq \delta h$ and thus $|a|^{-2}\in L^1$). 
On the other hand, since  $h\gtrsim |a|^{-2}$, the space $L^2(hdm)$ embeds continuously into $L^2(|a|^{-2}dm)$.  As a consequence, $P_+$ is bounded from $L^2(|a|^{-2}dm)$ to itself which implies 
that $|a|^{-2}$ (and equivalently $|a|^{2}$) satisfies the $(A_2)$ condition. 
\end{proof}

As a consequence of our discussion, we can deduce the following result:

\begin{Theorem}
Let $b\in H^\infty_1$ be  non-extreme, and $\mu\in M_+(\D^{-})$. Then the following are equivalent:
\begin{enumerate}
\item We have $\HH(b)_\mu=\HH(b)$ and $\|f\|_\mu\asymp \|f\|_b$ for any $f\in\HH(b)$;
\item The following conditions hold: 
\begin{enumerate}
\item the function $a$ is $\mu$-admissible, 
\item the pair $(a,b)$ is a corona pair,
\item the function $|a|^2$ satisfies $(A_2)$,
\item the measure $\nu$, defined by $d\nu=|a|^2\,d\mu$, satisfies
$$0 < \inf_{I}\frac{\nu\left(S(I)\right)}{m(I)} \leq \sup_{I}\frac{\nu\left(S(I)\right)}{m(I)} < \infty,$$
where the infimum and supremum above are taken over all open arcs $I$ of $\T$.
\end{enumerate}
\end{enumerate}
\end{Theorem}

\begin{Example}
Surely an example is important here. Let $a(z) :=  c_{\alpha} (1 - z)^{\alpha}$, where $\alpha \in (0, 1/2)$ and $c_{\alpha}$ is chosen so that $a \in H^{\infty}_{1}$. 
%Simple integral estimates, using the fact that 
As we have already mentioned earlier, since $0 < \alpha < 1/2$, the function $|a|^2$ satisfies the $(A_2)$ condition. Choose $b$ to be the outer function in $H^{\infty}_{1}$ satisfying $|a|^2 + |b|^2 = 1$ on $\T$. Standard theory, using the fact that $a$ is H\"{o}lder continuous on $\D^{-}$, will show that $b$ is continuous on $\D^{-}$ (see \cite{MR0289784}). It follows that $(a, b)$ is a corona pair. If $\sigma \in M_{+}(\D^{-})$ is any Carleson measure for $H^2$, then one can show that $d \mu := |a|^{-2} dm + d \sigma$ satisfies the conditions of the above theorem.
\end{Example}

We end this section with a proof  that if $b$ is non-extreme and non-constant then there are no isometric measures for $\HH(b)$. This requires a preliminary technical result already known  \cite{MR847333}. We provide a different proof which is slightly shorter but needs the additional assumption that $b/a\in H^2$. Recall from Section \ref{section2} that the polynomials belong to $\mathscr{H}(b)$ when $b$ is  non-extreme.

\begin{Lemma}\label{normz^n}
Let  $(c_k)_{k\ge 0}$ be the Taylor coefficients of the analytic function $b/a$ and assume that $b/a\in H^2$. Then
\[
\Vert z^n\Vert_{b}^2=1+\sum_{j=0}^{n}|c_j|^2.
\]
\end{Lemma}

\begin{proof}
By (\ref{normdBR}) and \eqref{Tba} we need to calculate $\Vert h_n\Vert_{2}$, where 
$$h_n=T_{\bar b/\bar a}z^n.$$ We have
\[
\|h_n\|_2^2=\left\|P_+ \frac{\bar b}{\bar a}z^n\right\|_2^2=\left\|P_- {\bar z}^{n+1}\frac{b}{a}\right\|_2^2=\left\|z^{n+1}P_- {\bar z}^{n+1}\frac{b}{a}\right\|_2^2=\sum_{j=0}^n |c_j|^2,
\]
where for the last identity, we used the fact that $z^{n+1}P_-{\bar z}^{n+1}$ is the orthogonal projection onto $(z^{n+1}H^2)^\perp$, that is the orthogonal projection onto the set of polynomials of degree at most $n$. 
\end{proof}

\begin{Theorem}
\label{no-isometry}
When $b$ is non-constant and non-extreme, there are no positive isometric measures for $\HH(b)$.
\end{Theorem}

\begin{proof}
Let us assume to the contrary that $b$ is non-constant and that there exists a $\mu\in M_{+}(\D^-)$ such that
\[
\|f\|_b=\|f\|_\mu,\quad f\in \HH(b).
\]
Let us apply this identity to $f  = z^{n}$.
First observe that
\[
\|z^n\|^2_\mu =\int_{\D^-}|z|^{2n}d\mu(z)=\mu(\T)+\int_{\D}|z|^{2n}d\mu(z).
\]
Lemma \ref{normz^n} yields
\begin{equation}
\label{isometriczn}
\mu(\T)+\int_{\D}|z|^{2n}d\mu(z) =\|z^n\|_b^2= 1+\sum_{j=0}^{n}|c_j|^2,\quad n\geq 0.
\end{equation}
Now let  $n\to \infty$ to get
\[
\mu(\T) = 1+\sum_{j=0}^{+\infty}|c_j|^2.
\]
Combine this identity with the one in  (\ref{isometriczn}) to obtain
\[
\int_{\D}|z|^{2n}d\mu(z)=0 
\]
and
\[
\sum_{j=n+1}^{+\infty}|c_j|^2 =0
\]
for all $n\geq 0$. In particular, the last identity for $n=0$ gives that  $b/a$ must be constant, or equivalently $b=ka$, with $k\in\mathbb{C}$. Hence, since 
$$1=|a|^2+|b|^2=|a|^2(1+|k|^2) \quad \mbox{a.e. on $\T$},$$
 $|a|^2$ is constant on $\T$. But since $a$ is outer, this forces $a$, and hence $b$, to be constant, yielding the desired contradiction. 
\end{proof}

\begin{Remark}
When $b$ is constant then $\mathscr{H}(b) = H^2$ with the norms differing by the constant $\sqrt{1 - |b|^2}$. In this case the only isometric measure for $H^2$ is Lebesgue measure $m$. Surely this is well-known but we include the following simple proof for the convenience of the reader. Indeed for each $n \in \N \cup \{0\}$ 
$$1 = \|z^{n}\|^2_{2} = \int_{\D} |z|^{2 n} d \mu + \mu(\T).$$ By the dominated convergence theorem, the first term on the right hand side goes to zero as $n \to \infty$ and so $\mu(\T) = 1$. But this means, by setting $n = 0$ in the previous equation, that $\mu(\D) = 0$. This implies that $\mu = \mu|\T$. By Carleson's criterion we see that $\mu \ll m$ and so $d \mu = h dm$. To see that $h$ is equal to one almost everywhere, apply the fact that $\mu$ is an isometric measure to the normalized reproducing kernels 
\begin{equation} \label{NRKH2}
\kappa_{\lambda}(z) := \frac{\sqrt{1 - |\lambda|^2}}{1 - \overline{\lambda} z}
\end{equation}
 to get that 
$$1 = \int_{\T} \frac{1 - |\lambda|^2}{|1 - \overline{\zeta} \lambda|^2} h(\zeta) dm(\zeta), \quad \lambda \in \D.$$ From basic facts about the Poisson integral it follows that $h \equiv 1$. Thus $\mu = m$.
\end{Remark}

\section{Final Remark}

Suppose that $b$ is an inner function. Then $\mathscr{H}(b) = (b H^{2})^{\perp}$ is the classical model space and is certainly a closed subspace of $H^2$. There is a concept developed in \cite{BFGHR} of a dominating set. Here a Borel set $E \subset \T$ with $m(E) < 1$ is called a \emph{dominating set} for $(b H^2)^{\perp}$ if 
$$\int_{\T} |f|^2 dm \lesssim \int_{E} |f|^2 dm, \quad f \in (b H^2)^{\perp}.$$ In \cite{BFGHR} it is shown that dominating sets exist for every model space $(b H^2)^{\perp}$ and can be used to give sufficient conditions for reverse Carleson embeddings for these spaces. 

One might be tempted to define a notion of dominating set for $\mathscr{H}(b)$ as a Borel set $E \subset \T$ satisfying 
$$\|f\|_{b}^{2} \lesssim \int_{E} |f|^2 dm, \quad f \in \mathscr{H}(b).$$ However, there is no real point to this. 

\begin{Proposition}
If $b \in H^{\infty}_{1}$ and not an inner function then
there is no Borel subset $E$ of $\T$ with $0 < m(E) < 1$ for which 
\begin{equation} \label{dominating}
\|f\|_{b}^{2} \lesssim \int_{E} |f|^2 dm, \quad f \in \mathscr{H}(b).
\end{equation}
\end{Proposition}

\begin{proof}
If $E$ were such a set satisfying \eqref{dominating} then $\T$ would also satisfy \eqref{dominating}. However, since the embedding of $\mathscr{H}(b)$ into $H^2$ is contractive, then we have 
$$\int_{\T} |f|^2 dm \le \|f\|_{b}^{2} \lesssim \int_{\T} |f|^2 dm, \quad f \in \mathscr{H}(b).$$
This means that $\mathscr{H}(b)$ is a closed subspace of $H^2$, which can only happen when either $b$ is an inner function or when $\|b\|_{\infty} < 1$ \cite[p.~10]{Sa}.

Since we are assuming that $b$ is not an inner function, we are left with dealing with the case $\|b\|_{\infty} < 1$. Here $\mathscr{H}(b) = H^2$ with an equivalent norm. By using the normalized kernel functions $\kappa_{\lambda}(z)$ from \eqref{NRKH2} in 
 the inequality \eqref{dominating} and basic facts about pointwise limits of Poisson integrals (and the fact that $m(E) < 1$) we get a contradiction. 
\end{proof}

\bibliographystyle{plain}

\bibliography{references}

\def\cprime{$'$}
\begin{thebibliography}{10}

\bibitem{Al98}
A.~B. Aleksandrov.
\newblock Isometric embeddings of co-invariant subspaces of the shift operator.
\newblock {\em Zap. Nauchn. Sem. S.-Peterburg. Otdel. Mat. Inst. Steklov.
  (POMI)}, 232(Issled. po Linein. Oper. i Teor. Funktsii. 24):5--15, 213, 1996.

\bibitem{Alek}
A.~B. Aleksandrov.
\newblock Embedding theorems for coinvariant subspaces of the shift operator.
  {II}.
\newblock {\em Zap. Nauchn. Sem. S.-Peterburg. Otdel. Mat. Inst. Steklov.
  (POMI)}, 262(Issled. po Linein. Oper. i Teor. Funkts. 27):5--48, 231, 1999.

\bibitem{Baranov-JFA05}
A.~D. Baranov.
\newblock Bernstein-type inequalities for shift-coinvariant subspaces and their
  applications to {C}arleson embeddings.
\newblock {\em J. Funct. Anal.}, 223(1):116--146, 2005.

\bibitem{BFM}
A.D. Baranov, E.~Fricain, and J.~Mashreghi.
\newblock Weighted norm inequalities for de {B}ranges-{R}ovnyak spaces and
  their applications.
\newblock {\em Amer. J. Math.}, 132(1):125--155, 2010.

\bibitem{BFGHR}
A.~Blandigneres, E.~Fricain, F.~Gaunard, A.~Hartmann, and W.~T. Ross.
\newblock Reverse {C}arleson embeddings for model spaces.
\newblock To appear, {J}. {L}ondon. {M}ath. {S}oc.

\bibitem{chacon}
G.~Chacon.
\newblock {\em Carleson-type inequalitites in harmonically weighted {D}irichlet
  spaces}.
\newblock PhD thesis, University of Tennessee, 2010.

\bibitem{fricain-shabankhah}
G.~Chacon, E.~Fricain, and M.~Shabankhah.
\newblock Carleson measures and reproducing kernel thesis in {D}irichlet-type
  spaces.
\newblock To appear, {S}t. {P}etersburg {M}athematical {J}ournal.

\bibitem{MR2514455}
Nicolas Chevrot, Emmanuel Fricain, and Dan Timotin.
\newblock On certain {R}iesz families in vector-valued de {B}ranges-{R}ovnyak
  spaces.
\newblock {\em J. Math. Anal. Appl.}, 355(1):110--125, 2009.

\bibitem{CRM}
J.~Cima, A.~Matheson, and W.~Ross.
\newblock {\em The {C}auchy transform}, volume 125 of {\em Mathematical Surveys
  and Monographs}.
\newblock American Mathematical Society, Providence, RI, 2006.

\bibitem{Clark72}
D.~N. Clark.
\newblock One dimensional perturbations of restricted shifts.
\newblock {\em J. Analyse Math.}, 25:169--191, 1972.

\bibitem{Cohn}
B.~Cohn.
\newblock Carleson measures for functions orthogonal to invariant subspaces.
\newblock {\em Pacific J. Math.}, 103(2):347--364, 1982.

\bibitem{CL}
E.~F. Collingwood and A.~J. Lohwater.
\newblock {\em The theory of cluster sets}.
\newblock Cambridge Tracts in Mathematics and Mathematical Physics, No. 56.
  Cambridge University Press, Cambridge, 1966.

\bibitem{ransford}
C.~Costara and T.~Ransford.
\newblock Why de {B}ranges--{R}ovnyak spaces are {D}irichlet spaces (and vice
  versa)?
\newblock Preprint.

\bibitem{MR0098981}
Karel de~Leeuw and Walter Rudin.
\newblock Extreme points and extremum problems in {$H_{1}$}.
\newblock {\em Pacific J. Math.}, 8:467--485, 1958.

\bibitem{Duo}
J.~Duoandikoetxea.
\newblock {\em Fourier analysis}, volume~29 of {\em Graduate Studies in
  Mathematics}.
\newblock American Mathematical Society, Providence, RI, 2001.
\newblock Translated and revised from the 1995 Spanish original by David
  Cruz-Uribe.

\bibitem{Duren}
P.~L. Duren.
\newblock {\em Theory of ${H}\sp{p}$ spaces}.
\newblock Academic Press, New York, 1970.

\bibitem{MR2159461}
E.~Fricain.
\newblock Bases of reproducing kernels in de {B}ranges spaces.
\newblock {\em J. Funct. Anal.}, 226(2):373--405, 2005.

\bibitem{Garnett}
J.~Garnett.
\newblock {\em Bounded analytic functions}, volume 236 of {\em Graduate Texts
  in Mathematics}.
\newblock Springer, New York, first edition, 2007.

\bibitem{HMNOC}
A.~Hartmann, X.~Massaneda, A.~Nicolau, and J.~Ortega-Cerd\`a.
\newblock Reverse {C}arleson measures in {H}ardy spaces.
\newblock Preprint.

\bibitem{MR0289784}
V.~P. Havin and F.~A. {\v{S}}amojan.
\newblock Analytic functions with a {L}ipschitzian modulus of the boundary
  values.
\newblock {\em Zap. Nau\v cn. Sem. Leningrad. Otdel. Mat. Inst. Steklov.
  (LOMI)}, 19:237--239, 1970.

\bibitem{Kacnelson}
V.~{\`E}. Kacnel{\cprime}son.
\newblock Equivalent norms in spaces of entire functions.
\newblock {\em Mat. Sb. (N.S.)}, 92(134):34--54, 165, 1973.

\bibitem{Queffelec}
Pascal Lef{\`e}vre, Daniel Li, Herv{\'e} Queff{\'e}lec, and Luis
  Rodr{\'{\i}}guez-Piazza.
\newblock Some revisited results about composition operators on {H}ardy spaces.
\newblock {\em Rev. Mat. Iberoam.}, 28(1):57--76, 2012.

\bibitem{Logvinenko}
V.~N. Logvinenko and Ju.~F. Sereda.
\newblock Equivalent norms in spaces of entire functions of exponential type.
\newblock {\em Teor. Funkci\u\i\ Funkcional. Anal. i Prilo\v zen.}, (Vyp.
  20):102--111, 175, 1974.

\bibitem{Luecking85}
D.~H. Luecking.
\newblock Forward and reverse {C}arleson inequalities for functions in
  {B}ergman spaces and their derivatives.
\newblock {\em Amer. J. Math.}, 107(1):85--111, 1985.

\bibitem{Luecking88}
D.~H. Luecking.
\newblock Dominating measures for spaces of analytic functions.
\newblock {\em Illinois J. Math.}, 32(1):23--39, 1988.

\bibitem{MR1945290}
F.~Nazarov and A.~Volberg.
\newblock The {B}ellman function, the two-weight {H}ilbert transform, and
  embeddings of the model spaces {$K_\theta$}.
\newblock {\em J. Anal. Math.}, 87:385--414, 2002.
\newblock Dedicated to the memory of Thomas H. Wolff.

\bibitem{Nik}
N.~K. Nikolski.
\newblock {\em Operators, functions, and systems: an easy reading. {V}ol. 1},
  volume~92 of {\em Mathematical Surveys and Monographs}.
\newblock American Mathematical Society, Providence, RI, 2002.
\newblock Hardy, Hankel, and Toeplitz, Translated from the French by Andreas
  Hartmann.

\bibitem{MR1864396}
N.K. Nikolski.
\newblock {\em Operators, functions, and systems: an easy reading. {V}ol. 1},
  volume~92 of {\em Mathematical Surveys and Monographs}.
\newblock American Mathematical Society, Providence, RI, 2002.
\newblock Hardy, Hankel, and Toeplitz, Translated from the French by Andreas
  Hartmann.

\bibitem{Panejah62}
B.~P. Panejah.
\newblock On some problems in harmonic analysis.
\newblock {\em Dokl. Akad. Nauk SSSR}, 142:1026--1029, 1962.

\bibitem{Panejah66}
B.~P. Panejah.
\newblock Certain inequalities for functions of exponential type and a priori
  estimates for general differential operators.
\newblock {\em Uspehi Mat. Nauk}, 21(3 (129)):75--114, 1966.

\bibitem{MR2417425}
S.~Petermichl, S.~Treil, and B.D. Wick.
\newblock Carleson potentials and the reproducing kernel thesis for embedding
  theorems.
\newblock {\em Illinois J. Math.}, 51(4):1249--1263, 2007.

\bibitem{PolSa}
A.~Poltoratski and D.~Sarason.
\newblock Aleksandrov-{C}lark measures.
\newblock In {\em Recent advances in operator-related function theory}, volume
  393 of {\em Contemp. Math.}, pages 1--14. Amer. Math. Soc., Providence, RI,
  2006.

\bibitem{MR847333}
D.~Sarason.
\newblock Doubly shift-invariant spaces in {$H^2$}.
\newblock {\em J. Operator Theory}, 16(1):75--97, 1986.

\bibitem{Sa}
D.~Sarason.
\newblock {\em Sub-{H}ardy {H}ilbert spaces in the unit disk}.
\newblock University of Arkansas Lecture Notes in the Mathematical Sciences,
  10. John Wiley \& Sons Inc., New York, 1994.
\newblock A Wiley-Interscience Publication.

\bibitem{MR2418122}
D.~Sarason.
\newblock Unbounded {T}oeplitz operators.
\newblock {\em Integral Equations Operator Theory}, 61(2):281--298, 2008.

\bibitem{Volberg-81}
A.~L. Vol{\cprime}berg.
\newblock Thin and thick families of rational fractions.
\newblock In {\em Complex analysis and spectral theory ({L}eningrad,
  1979/1980)}, volume 864 of {\em Lecture Notes in Math.}, pages 440--480.
  Springer, Berlin, 1981.

\bibitem{TV}
A.~L. Vol{\cprime}berg and S.~R. Treil{\cprime}.
\newblock Embedding theorems for invariant subspaces of the inverse shift
  operator.
\newblock {\em Zap. Nauchn. Sem. Leningrad. Otdel. Mat. Inst. Steklov. (LOMI)},
  149(Issled. Linein. Teor. Funktsii. XV):38--51, 186--187, 1986.

\end{thebibliography}

\end{document}